\documentclass{gtmon_a}
\pdfoutput=1

\usepackage{amscd}
\usepackage{stmaryrd}

\proceedingstitle{Proceedings of the Nishida Fest (Kinosaki 2003)}
\conferencestart{28 July 2003}
\conferenceend{8 August 2003}
\conferencename{International Conference in Homotopy Theory}
\conferencelocation{Kinosaki, Japan}

\editor{Matthew Ando}
\givenname{Matthew}
\surname{Ando}

\editor{Norihiko Minami}
\givenname{Norihiko}
\surname{Minami}

\editor{Jack Morava}
\givenname{Jack}
\surname{Morava}

\editor{W Stephen Wilson}
\givenname{W Stephen}
\surname{Wilson}

\title{On excess filtration on the Steenrod algebra}

\author{Atsushi Yamaguchi}
\givenname{Atsushi}
\surname{Yamaguchi}
\address{Osaka Prefecture University\\\newline
Sakai\\
Osaka\\
Japan}
\email{yamaguti@las.osakafu-u.ac.jp}
\urladdr{http://www.las.osakafu-u.ac.jp/~yamaguti/}

\volumenumber{10}
\issuenumber{}
\publicationyear{2007}
\papernumber{22}
\startpage{423}
\endpage{449}

\doi{}
\MR{}
\Zbl{}

\arxivreference{}


\received{31 August 2004}
\revised{5 April 2006}
\accepted{}
\published{18 April 2007}
\publishedonline{18 April 2007}
\proposed{}
\seconded{}
\corresponding{}
\version{}

\makeatletter
\def\cnewtheorem#1[#2]#3{\newtheorem{#1}{#3}[section]
\expandafter\let\csname c@#1\endcsname\c@thm}

\let\xysavmatrix\xymatrix
\def\xymatrix{\disablesubscriptcorrection\xysavmatrix}
\AtBeginDocument{\let\bar\wbar\let\tilde\wtilde\let\hat\what}

\newtheorem{thm}{Theorem}[section]
\cnewtheorem{lem}[thm]{Lemma}
\cnewtheorem{cor}[thm]{Corollary}
\cnewtheorem{prop}[thm]{Proposition}

\theoremstyle{definition}
\cnewtheorem{defn}[thm]{Definition}
\cnewtheorem{cond}[thm]{Condition}
\cnewtheorem{rem}[thm]{Remark}

\makeatother  

\def\bff{{\mathbb{F}}}
\def\bz{{\mathbb{Z}}}
\def\bzr{{\mathbf{0}}}
\def\ca{\mathcal A}
\def\ce{\mathcal E}
\def\cf{\mathcal F}
\def\cm{\mathcal M}
\def\cco{\mathcal O}
\def\cu{\mathcal U}
\def\fF{\mathfrak F}
\def\alg{{\mathcal A}{\it lg}}
\def\Cok{\operatorname{Coker}}
\def\gr{{\mathcal G}{\it r}}
\def\hom{\operatorname{Hom}}
\def\Im{\operatorname{Im}}
\def\Ker{\operatorname{Ker}}
\def\ndiv{\!\!\not |}
\def\ob{\operatorname{Ob}}
\def\seq{\operatorname{Seq}}
\def\Sq{{\textsl Sq}}

\makeop{id}

\begin{document}

\begin{asciiabstract}
In this note, we study some properties of the filtration of the Steenrod 
algebra defined from the excess of admissible monomials. 
We give several conditions on a cocommutative graded Hopf algebra A* 
which enable us to develop the theory of unstable A*-modules.
\end{asciiabstract}

\begin{abstract}
In this note, we study some properties of the filtration of the Steenrod 
algebra defined from the excess of admissible monomials. 
We give several conditions on a cocommutative graded Hopf algebra $A^*$ 
which 
enable us to develop the theory of unstable $A^*$--modules. 
\end{abstract}

\maketitle

\section*{Introduction}

The theory of unstable modules over the Steenrod algebra has been 
developed by 
many researchers and has various geometric applications. (See 
Schwartz \cite{Sc} and its references.) 
It was so successful that it might be interesting to consider the 
structure of 
the Steenrod algebra which enable us to define the notion of unstable 
modules. 
Let us call the filtration on the Steenrod algebra defined from the excess 
of 
admissible monomials the excess filtration. 
(See \fullref{excess filtration} below.)
We note that this filtration plays an essential role in developing the 
theory of unstable modules. 

The aim of this note is to give several conditions on filtered graded Hopf 
algebra $A^*$ which allows us to deal the theory of unstable $A^*$--modules 
axiomatically. 
In the first and second sections, we study properties of the excess 
filtration 
on the Steenrod algebra $\ca_p$. 
In \fullref{sec:sec3}, we propose nine conditions on a decreasing filtration on a 
cocommutative graded Hopf algebra $A^*$ over a field which may suffice to 
develop the theory of unstable modules. 
We also verify several facts (eg \fullref{SES}, \fullref{ker coker}, 
\fullref{left adjoint}) which are known to hold for the case of the 
Steenrod 
algebra. 
To give an example of a filtered Hopf algebra other than the Steenrod 
algebra, 
we consider the group scheme defined from the unipotent matrix groups in 
\fullref{sec:sec4}. 
We embed the group scheme represented by the dual Steenrod algebra as a 
closed subscheme of infinite dimensional unipotent group scheme 
represented by 
a certain Hopf algebra $A_{(p)*}$ which has a filtration satisfying the 
dual 
of first six conditions given in \fullref{sec:sec3}. 
We observe that this filtration induces the filtration on the mod $p$ dual 
Steenrod algebra $\ca_{p*}$ which is the dual of the excess filtration. 
In \fullref{sec:sec5} we show that the affine group scheme represented by 
$\ca_{p*}$ 
is naturally equivalent to a $\bff_p$--group functor which assigns to an
$\bff_p$--algebra $R^*$ a certain subgroup of the strict isomorphisms of the 
additive 
formal group law over $R^*[\varepsilon]/(\varepsilon^2)$ in \fullref{sec:sec4}. 

\section{Basic properties of excess filtration} \label{sec:sec1}

We denote by $\ca_p$ the mod $p$ Steenrod algebra and by $\ca_{p*}$ its 
dual. 
Let $\seq$ be the set of all infinite sequences  
$(i_1,i_2,\dots,i_n,\dots)$ of non-negative integers such that $i_n=0$ for 
all 
but finite number of $n$. 
Let $\seq^o$ be a subset of $\seq$ consisting of sequences 
$(i_1,i_2,\dots,i_n,\dots)$ such that $i_k=0,1$ if $k$ is odd. 
If $i_n=0$ for $n>N$, we denote $(i_1,i_2,\dots,i_n,\dots)$ by 
$(i_1,i_2,\dots,i_N)$. 

\begin{defn}[Steenrod--Epstein \cite{SE}]
\label{P^I}
For $I=(\varepsilon_0,i_1,\varepsilon_1,\dots,i_n,\varepsilon_n)\in\seq^o$ 
and an odd prime $p$, we put 
$$d_p(I)=2(p-1)\sum_{s=1}^ni_s+\sum_{s=0}^n\varepsilon_s,\qquad
e_p(I)=\sum_{s=0}^n\varepsilon_s+2\sum_{s=1}^n(i_s-pi_{s+1}-\varepsilon_s).$$
For $J=(j_1,j_2,\dots,j_n)\in\seq$, we put  
$$d_2(J)=\sum_{s=1}^nj_s,\qua e_2(J)=\sum_{s=1}^n(j_s-2j_{s+1}).$$ 
Then
\begin{align*}
\wp^I&=\beta^{\varepsilon_0}\wp^{i_1}\beta^{\varepsilon_1}
\wp^{i_2}\beta^{\varepsilon_2}\cdots\wp^{i_n}\beta^{\varepsilon_n}
\in\ca_p^{d_p(I)} \\[-1ex]
\text{and}\quad
\Sq^J&=\Sq^{j_1}\Sq^{j_2}\dots\Sq^{j_n}\in\ca_2^{d_2(J)}.
\end{align*}
We call $d_p(I)$ the degree of $I$ and $e_p(I)$ the excess of $I$. 
\end{defn}

\begin{prop}\label{d,e}
Suppose $I=(\varepsilon_0,i_1,\varepsilon_1,\dots,i_n,\varepsilon_n,\dots)
\in\seq^o$ and $$J=(j_1,j_2,\dots,j_n,\dots)\in\seq.$$ 
\begin{enumerate}
\item[\rm(1)] $e_p(I)=2pi_1+2\varepsilon_0-d_p(I)$ if $p$ is an odd prime,

$e_2(J)=2j_1-d_2(J)$.
\item[\rm(2)] $e_p(I)\le2i_1+\varepsilon_0$ and the equality holds if
and only if $I=(\varepsilon_0,i_1)$ for an odd prime $p$. 

$e_2(J)\le j_1$ and the equality holds if and only if $J=(j_1)$. 
\item[\rm(3)] $d_p(I)\ge(p-1)e_p(I)-\varepsilon_0(p-2)$ and the equality
holds if and only if $I=(\varepsilon_0,i_1)$ for an odd prime $p$. 

$d_2(J)\ge e_2(J)$ and the equality holds if and only if $J=(j_1)$. 
\end{enumerate}
\end{prop}

\begin{proof}
(1)\qua These are direct consequences of the definitions of $d_p(I)$ and 
$e_p(I)$.

(2)\qua For an odd prime $p$, 
$e_p(I)=2i_1+\varepsilon_0-2(p-1)\sum_{s=2}^n
i_s-\sum_{s=1}^n\varepsilon_s\leqq2i_1+\varepsilon_0$. 
If $p=2$, $e_2(J)=j_1-\sum_{s=2}^nj_s\leqq j_1$.

(3)\qua If $p$ is an odd prime, $e_p(I)\le2i_1+\varepsilon_0$ is equivalent to 
$e_p(I)-2pi_1-2\varepsilon_0\le-(p-1)e_p(I)+\varepsilon_0(p-2)$. 
Then the assertion follows from (1). 

The proof of the case $p=2$ is similar.
\end{proof}

\begin{cor}\label{d<e}
Let $j$ be a fixed non-negative integer, $\varepsilon=0,1$ and 
$I\in\seq^o$, 
$J\in\seq$. 
\begin {enumerate}
\item[\rm(1)] Suppose that $p$ is an odd prime. 
If $e_p(I)\ge 2j+\varepsilon$, then $d_p(I)\ge2j(p-1)+\varepsilon$ and the 
equality holds if and only if $I=(\varepsilon,j)$. 
\item[\rm(2)] If $e_2(J)\ge j$, then $d_2(J)\ge j$ and the equality
holds if and only if $J=(j)$.
\end{enumerate}
\end{cor}

\begin{proof}
(1)\qua Assume that $e_p(I)\ge 2j$ and $d_p(I)\le2j(p-1)-1$. 
By \fullref{d,e},
$$2i_1+\varepsilon_0-2pi_1-\varepsilon_0\ge 
e_p(I)-2pi_1-\varepsilon_0=-d_p(I)\ge-2j(p-1)+1.$$ 
Hence $2j(p-1)\ge2i_1(p-1)+1$, which implies $j\ge i_1+1$. 

Then $2i_1+\varepsilon_0\ge e_p(I)\ge2j\ge2i_1+2$ but this contradicts 
$\varepsilon_0\le1$. 
Therefore $d_p(I)\ge2j(p-1)$. 

Suppose $e_p(I)\ge 2j$ and $d_p(I)=2j(p-1)$. 
Since $d_p(I)\ge2i_1(p-1)$, we have $j\ge i_1$. 
On the other hand, since $2j\le e_p(I)\le 2i_1+\varepsilon_0$, we have 
$j\le i_1$. 
Hence $j=i_1$ and this implies $i_s=0$ for $s\ge2$ and $\varepsilon_s=0$ 
for 
$s\ge0$. 

Assume $e_p(I)\ge 2j+1$. 
By \fullref{d,e}, 
$$d_p(I)\ge(p-1)e_p(I)-\varepsilon_0(p-2)\ge(p-1)(2j+1)-p+2=
2j(p-1)+1.$$ 
Suppose that $e_p(I)\ge 2j+1$ and $d_p(I)=2j(p-1)+1$. 
We have
$$d_p(I)\ge(p-1)e_p(I)-\varepsilon_0(p-2)\ge2j(p-1)+1=d_p(I).$$ 
Hence $I$ is of the form $(\varepsilon_0,i_1)$ by \fullref{d,e}. 
Then $d_p(I)=2i_1(p-1)+\varepsilon_0$ which equals to $2j(p-1)+1$.
Therefore $I=(1,j)$. 

(2)\qua The proof is similar as above.
\end{proof}

\begin{defn}[Steenrod--Epstein \cite{SE}]
We say  
$I=(\varepsilon_0,i_1,\varepsilon_1,\dots,i_n,\varepsilon_n,\dots)
\in\seq^o$ is ($p$--)admissible if $p$ is an odd prime and 
$i_s\ge pi_{s+1}+\varepsilon_s$ for $s=1,2,\dots$. 
For $p=2$, we say that $I=(i_1,i_2,\dots,i_n,\dots)\in\seq$ is 
(2--)admissible 
if $i_s\ge 2i_{s+1}$ for $s=1,2,\dots$. 
We denote by $\seq_p$ the subset of $\seq$ consisting of $p$--admissible 
sequences. 
\end{defn}

We quote the following fundamental results for later use. 

\begin{thm}[Steenrod--Epstein \cite{SE}]\quad
\label{admissible generation}
\begin{enumerate}
\item[\rm(1)] If $I\in\seq^o$, $\wp^I=\sum_{k=1}^lc_k\wp^{I_k}$ for 
some $c_k\in\bff_p$ and $I_k\in\seq_p$ such that $e_p(I_k)\ge e_p(I)$ for 
all $k=1,2,\dots,l$. 

Similarly, if $I\in\seq$, $\Sq^I=\sum_{k=1}^lc_k\Sq^{I_k}$ for some 
$c_k\in\bff_2$ and $I_k\in\seq_2$ such that $e_2(I_k)\ge e_2(I)$ for all 
$k=1,2,\dots,l$. 
\item[\rm(2)] $\{\wp^I|\,I\in\seq_p\}$ is a basis of $\ca_p$ if $p$ is an odd prime 
and 
$\{\Sq^I|\,I\in\seq_2\}$ is a basis of $\ca_2$. 
\end{enumerate}
\end{thm}

Let $\tau_n\in(\ca_{p*})^{2p^n-1}$, $\xi_n\in(\ca_{p*})^{2p^n-2}$ and 
$\zeta_n\in(\ca_{2*})^{2^n-1}$ be the elements given by Milnor \cite{Mi}. 
Recall that 
$\ca_{p*}=E(\tau_0,\tau_1,\dots)\otimes\bff_p[\xi_1,\xi_2,
\dots]$ if $p\ne2$, and $\ca_{2*}=\bff_2[\zeta_1,\zeta_2,\dots]$. 

Let $\seq^b$ be a subset of $\seq$ consisting of all sequences 
$(\varepsilon_0,\varepsilon_1,\dots,\varepsilon_n,\dots)$ such that 
$\varepsilon_n=0,1$ for all $n=0,1,\dots$. 

For $E=(\varepsilon_0,\varepsilon_1,\dots,\varepsilon_m)\in\seq^b$ and 
$R=(r_1,r_2,\dots,r_n)\in\seq$, we put
$$\tau(E)=\tau_0^{\varepsilon_0}
\tau_1^{\varepsilon_1}\cdots\tau_m^{\varepsilon_m},\quad
\xi(R)=\xi_1^{r_1}\xi_2^{r_2}\cdots\xi_n^{r_n} \quad\text{and}\quad
\zeta(R)=\zeta_1^{r_1}\zeta_2^{r_2}\cdots\zeta_n^{r_n}$$
as in \cite{Mi}.  Then, the Milnor basis is defined as follows. 

\begin{defn}[Milnor \cite{Mi}]
\label{Milnor basis}
We denote by $\wp(S)$ the dual of $\xi(S)$ with respect to the basis 
$\{\tau(E)\xi(R)|E\in\seq^b,R\in\seq\}$ is of $\ca_{p*}$ if $p\ne2$ 
and by 
$\Sq(S)$ the dual of $\zeta(S)$ with respect to the basis 
$\{\zeta(R)|\,R\in\seq\}$ of $\ca_{2*}$. 
If $p$ is odd, let $Q_n$ be the dual of $\tau_n$ with respect to the basis 
$\{\tau(E)\xi(R)|\,E\in\seq^b,\;R\in\seq\}$. 
Put $Q(E)=Q_0^{\varepsilon_0}Q_1^{\varepsilon_1}\cdots 
Q_n^{\varepsilon_n}$ 
for $E=(\varepsilon_0,\varepsilon_1,\dots,\varepsilon_n)\in\seq^b$. 
\end{defn}

\begin{defn}\label{excess filtration}
Let $F_i\ca_p$ be the subspace of $\ca_p$ spanned by 
$$\bigl\{\wp^I\big|I\in\seq_p,e_p(I)\ge 
i\bigr\}\text{ if }
p\ne2,\bigl\{\Sq^I\big|\,I\in\seq_2,e_2(I)\ge i\bigr\}
\text{ if }p=2.$$
Thus we have an decreasing filtration $\fF_p=(F_i\ca_p)_{i\in\bz}$ on 
$\ca_p$. 
We call $\fF_p$ the excess filtration. 
\end{defn}

Clearly, $\fF_p$ satisfies the following. 
\begin{enumerate}
\item {(E1)} $F_i\ca_p=\ca_p$ if $i\le0$. 
\item {(E2)} $\bigcap\limits_{i\in\bz}F_i\ca_p=\{0\}$.
\end{enumerate} 

The next result is a direct consequence of \fullref{admissible 
generation}. 

\begin{prop}\label{fil 2}
For $I\in\seq^o$, $\wp^I\in F_i\ca_p$ if $p$ is an odd prime and 
$e_p(I)\ge i$. 
For $I\in\seq$, $\Sq^I\in F_i\ca_2$ if $e_2(I)\ge i$. 
\end{prop}

The following properties of the excess filtration are a sort of 
``folklore". 

\begin{prop}\label{fil 3} 
Let $\mu \co \ca_p\otimes\ca_p\to\ca_p$ and 
$\delta \co \ca_p\to\ca_p\otimes\ca_p$ 
be 
the product and the coproduct of $\ca_p$. 
Then $\fF_p$ satisfies the following conditions. 
\begin{enumerate}
\item {(E3)} $F_i\ca_p$ are left ideals of $\ca_p$ for $i\in\bz$. 
\item {(E4)} $\mu(F_i\ca_p\otimes\ca_p^j)\subset F_{i-j}\ca_p$ for 
$i,j\in\bz$. 
\item {(E5)} $\delta(F_i\ca_p)\subset
\sum_{j+k=i}F_j\ca_p\otimes F_k\ca_p$ for $i\in\bz$. 
\end{enumerate}
\end{prop}

\begin{proof}
Let $I=(\varepsilon_0,i_1,\varepsilon_1,\dots,i_n,\varepsilon_n)$ be a 
sequence belonging to $\seq_p$ such that $e_p(I)\ge i$ and $i_n\ge1$ if 
$n\ge1$. 

If $\varepsilon_0=1$, then $\beta\wp^I=0$. 
If $\varepsilon_0=0$, the excess of 
$(1,i_1,\varepsilon_1,\dots,i_n,\varepsilon_n)$ is bigger than $e_p(I)$. 
Hence $\beta\wp^I\in F_i\ca_p$. 
If $j\ge pi_1+\varepsilon_0$, then 
$(0,j,\varepsilon_0,i_1,\varepsilon_1,\dots,i_n,
\varepsilon_n)$ is admissible and its excess is not less than $e_p(I)$. 
Hence $\wp^j\wp^I\in F_i\ca_p$ in this case. 
Suppose $1\le j<pi_1+\varepsilon_0$. 
Then, by \fullref{admissible generation}, $\wp^j\wp^I$ is a linear 
combination 
of $\wp^J$'s such that $e_p(J)\ge e_p(I)$ and $J\in\seq_p$. 
Thus $\fF_p$ satisfies (E3). 

If $\varepsilon_n=1$, then $\wp^I\beta=0$. 
If $\varepsilon_n=0$, then 
$e_p(\varepsilon_0,i_1,\varepsilon_1,\dots,i_n,1)=e_p(I)-1$. 
Hence $\wp^I\beta\in F_{i-1}\ca_p$ by \fullref{fil 2}. 
If $n\ge1$, then 
$e_p(\varepsilon_0,i_1,\varepsilon_1,\dots,i_n,\varepsilon_n,j)=
2pi_1+2\varepsilon_0-d_p(I)-2j(p-1)=e_p(I)-2j(p-1)\ge i-2j(p-1)$. 
Hence $\wp^I\wp^j\in F_{i-2j(p-1)}\ca_p$ by \fullref{fil 2}. 
It is clear that $\wp^I\wp^j\in F_{i-2j(p-1)}\ca_p$ if $n=0$. 
Thus $\fF_p$ satisfies (E4). 

By the Cartan formula, $\delta(\wp^I)=\sum_{J+L=I}\wp^J\otimes\wp^L$. 
Put $J=(\alpha_0,j_1,\dots)$, $L=(\beta_0,l_1,\dots)$. 
Then, $e_p(J)+e_p(L)=2p(j_1+l_1)+2(\alpha_0+\beta_0)-d_p(J)-d_p(L)=
2pi_1+2\varepsilon_0-d_p(I)=e_p(I)$. 
Hence (E5) follows from \fullref{fil 2}. 
\end{proof}

Consider the dual filtration $\fF_p^*=(F_i\ca_{p*})_{i\in\bz}$ on 
$\ca_{p*}$, that is, $(F_i\ca_{p*})^n$ is the kernel of 
$$\kappa_{i+1}^* \co (\ca_{p*})^n=\hom(\ca_p^n,\bff_p)
\to\hom((F_{i+1}\ca_p)^n,\bff_p),$$ 
where $\kappa_i \co F_i\ca_p\to \ca_p$ is the inclusion map. 
The following is the dual of (E5) of \fullref{fil 3}. 

\begin{prop}\label{dual of c3} 
Let $\delta^* \co \ca_{p*}\otimes\ca_{p*}\to\ca_{p*}$ be the product of 
$\ca_{p*}$. 
Then $\fF_{p*}$ satisfies the following. 
\begin{enumerate}
\item {(E5${}^*\!$)\,}$\delta^*(F_j\ca_{p*}\otimes F_k\ca_{p*})\subset 
F_{j+k}\ca_{p*}$ for $j,k\in\bz$. 
\end{enumerate}
\end{prop}

We set $J_n=(0,p^{n-1},0,p^{n-2},\dots,0,1,0)$ and 
$J'_n=(0,p^{n-1},0,p^{n-2},\dots,0,1,1)$ for an odd prime $p$, 
$K_n=(2^{n-1},2^{n-2},\dots,2,1,)$. 
Then, $J_n$'s and $J'_n$'s are admissible and $d_p(J_n)=2p^n-2$, 
$d_p(J'_n)=2p^n-1$, $e_p(J_n)=2$, $e_p(J'_n)=1$ if $p$ is an odd prime, 
$d_2(K_n)=2^n-1$, $e_2(K_n)=1$. 

For $R=(r_1,r_2,\dots,r_n,\dots)\in\seq$, we put 
$|R|=\sum_{i\ge1}r_i$. 

\begin{prop}\label{dual filt2}
$\tau(E)\xi(R)\in F_{|E|+2|R|}\ca_{p*}-F_{|E|+2|E|-1}\ca_{p*}$ for 
$R\in\seq$ and $E\in\seq^b$, if $p$ is an odd prime. 
$\zeta(R)\in F_{|R|}\ca_{2*}-F_{|R|-1}\ca_{2*}$ for $R\in\seq$. 
\end{prop}

\begin{proof}
Since $e_p(J'_n)=1$ and $e_p(J_n)=2$, it follows from  
Milnor \cite[Lemma 8]{Mi} that $\tau_i\in F_1\ca_{p*}$ and $\xi_i\in F_2\ca_{p*}$. 
Similarly, since $e_2(K_n)=1$, we have $\zeta_i\in F_1\ca_{2*}$. 
Hence, by \fullref{dual of c3}, we have $\tau(E)\xi(R)\in 
F_{|E|+2|R|}\ca_{p*}$  if $p\ne2$, $\zeta(R)\in F_{|R|}\ca_{2*}$. 
On the other hand, it follows from \fullref{d,e2} and \cite[Lemma 8]{Mi} 
that 
$\tau(E)\xi(R)\not\in F_{|E|+2|E|-1}\ca_{p*}$ and $\zeta(R)\not\in 
F_{|R|-1}\ca_{2*}$. 
\end{proof}

We define the maps $\varrho_p \co \seq^o\to\seq_p$ and $\varrho_2 \co \seq\to\seq_2$ 
as follows. 
For 
$J=(\varepsilon_0,j_1,\varepsilon_1,\dots,j_n,\varepsilon_n)\in\seq^o$, 
put
$$i_s=\sum_{k=s}^n(\varepsilon_k+j_k)p^{k-s} \quad(s=1,2,\dots,n)$$ 
and 
$\varrho_p(J)=(\varepsilon_0,i_1,\varepsilon_1,\dots,i_n,\varepsilon_n)$.

If $p=2$, for $J=(j_1,j_2,\dots,j_n)\in\seq$, put 
$$i_s=\sum_{k=s}^nj_k2^{k-s} \quad(s=1,2,\dots,n)$$
and $\varrho_2(J)=(i_1,i_2,\dots,i_n)$. 

The following Lemmas are straightforward. 

\begin{lem}\label{varrho}
$\varrho_p$ is bijective and its inverse $\varrho_p^{-1}$ is given as 
follows. 
If $p$ is an odd prime, 
$\varrho_p^{-1}(\varepsilon_0,i_1,\varepsilon_1,\dots,i_n,\varepsilon_n)=
(\varepsilon_0,j_1,\varepsilon_1,\dots,j_n,\varepsilon_n)$, where 
$j_s=i_s-pi_{s+1}-\varepsilon_s$ (for $s=1,2,\dots,n-1$), and
$j_n=i_n-\varepsilon_n$. 

Similarly, $\varrho_2^{-1}(i_1,i_2,\dots,i_n)=(j_1,j_2,\dots,j_n)$, where 
$j_s=i_s-2i_{s+1}$ (for $s=1,2,\dots,n-1$), and $j_n=i_n$. 
\end{lem}

\begin{lem}\label{d,e2}
If $p\ne2$, for 
$J=(\varepsilon_0,j_1,\varepsilon_1,\dots,j_n,\varepsilon_n)
\in\seq^o$, we have 
$$d_p(\varrho_p(J))=\sum\limits_{k=1}^n2j_k(p^k-1)
+\sum\limits_{k=0}^n\varepsilon_k(2p^k-1)\text{ and }
e_p(\varrho_p(J))=2\sum\limits_{k=1}^nj_k+\sum\limits_{k=0}^n\varepsilon_k.
$$
If $p=2$, $J=(j_1,j_2,\dots,j_n)\in\seq$, we have  
$$d_2(\varrho_2(J))=\sum\limits_{k=1}^nj_k(2^k-1)\text{ and } 
e_2(\varrho_2(J))=\sum\limits_{k=1}^nj_k.$$
\end{lem}

\begin{prop}\label{dual filt3}
$\{\tau(E)\xi(R)| E\in\seq^b, R\in\seq, |E|+2|R|\le i\}$ is a basis of 
$F_i\ca_{p*}$ and 
$\{\zeta(R)| R\in\seq,|R|\le i\}$ is a basis of $F_i\ca_{2*}$. 
\end{prop}

\begin{proof}
Since $(F_i\ca_{p*})^n$ is isomorphic to 
$\hom^*((\ca_p/F_{i+1}\ca_p)^n,\bff_p)$, we have 
$$\dim(F_i\ca_{p*})^n=\dim(\ca_p/F_{i+1}\ca_p)^n=
\dim\ca_p^n-\dim(F_{i+1}\ca_p)^n.$$ 
Suppose $p$ is odd. 
By \fullref{admissible generation} (2), \fullref{varrho} and 
\fullref{d,e2}, 
$\dim\ca_p^n$ is the number of elements of a subset $S_n$ of $\seq^o$ 
defined 
by 
$$S_n=\biggl\{(\varepsilon_0,j_1,\varepsilon_1,\dots,j_n,
\varepsilon_n,\dots)\in
\seq^o\bigg|\sum_{k\ge0}(2p^k-1)\varepsilon_k+ \sum_{k\ge1}2(p^k-1)j_k = n
\biggr\}$$ 
and $\dim(F_{i+1}\ca_p)^n$ is the number of elements of 
$$\biggl\{(\varepsilon_0,j_1,\varepsilon_1,\dots,j_n,\varepsilon_n,\dots)\in
S_n\bigg|\sum_{k\ge0}\varepsilon_k+\sum_{k\ge1}2j_k\ge 
i+1\biggr\}.$$
Hence $\dim(F_i\ca_{p*})^n$ is the number of elements of 
$$\biggl\{(\varepsilon_0,j_1,\varepsilon_1,\dots,j_n,\varepsilon_n,\dots)\in
S_n\bigg|\sum_{k\ge0}\varepsilon_k+\sum_{k\ge1}2j_k\le 
i\biggr\},$$
which coincides with the number of elements of 
$$\{\tau(E)\xi(R)|\,E\in\seq^b,R\in\seq,|E|+2|R|\le i\}.$$
Therefore the assertion follows from \fullref{dual filt2}. 
The proof for the case $p=2$ is similar. 
\end{proof}

The following is shown by Kraines \cite{Kr} but is also a direct 
consequence of Milnor \cite[Theorem 4a]{Mi}, \fullref{dual filt2} and 
\fullref{dual filt3}. 

\begin{prop}[Kraines \cite{Kr}]\quad
\label{Milnor's basis2}
\begin{enumerate}
\item[\rm(1)] $Q(E)\wp(R)\in F_{|E|+2|R|}\ca_p-F_{|E|+2|R|+1}\ca_p$ for $R\in\seq$ 
and $E\in\seq^b$ if $p$ is an odd prime.  
$\Sq(R)\in F_{|R|}\ca_2-F_{|R|+1}\ca_2$ for $R\in\seq$. 
\item[\rm(2)] $\{Q(E)\wp(R)|\,E\in\seq^b,R\in\seq,|E|+2|R|\ge i\}$ is a basis of 
$F_i\ca_p$ for an odd prime $p$. 
$\{\Sq(R)|\,R\in\seq,|R|\ge i\}$ is a basis of $F_i\ca_2$. 
\end{enumerate}
\end{prop}

We set $E_i^j\ca_p=(F_i\ca_p)^j/(F_{i+1}\ca_p)^j$. 

\begin{prop}\label{fil 1}
Let $i$ be a non-negative integer and $\varepsilon=0$ or $1$. 
\begin{enumerate}
\item[\rm(1)] $(F_{2i+\varepsilon}\ca_p)^k=\{0\}$ for $k<2i(p-1)+\varepsilon$. 
\item[\rm(2)] If $p$ is an odd prime, 
$(F_{2i+\varepsilon}\ca_p)^{2i(p-1)+\varepsilon}$ 
is a one dimensional vector space spanned by $\beta^{\varepsilon}\wp^i$. 
$(F_i\ca_2)^i$ is a one dimensional vector space spanned by $\Sq^i$. 
\item[\rm(3)] $E_i^j\ca_p=\{0\}$ if $i+j\not\equiv0,2$ modulo $2p$. 
\end{enumerate}
\end{prop}

\begin{proof}
(1) and (2) are direct consequences of \fullref{d<e}. 

Suppose that $E=(\varepsilon_0,\varepsilon_1,\dots)\in\seq^b$, and 
$R=(r_1,r_2,\dots)\in\seq$ satisfy $|E|+2|R|=i$ and 
$Q(E)\wp(R)\in\ca_p^j$. 
Then 
$$i+j=\sum\limits_{s\ge0}2\varepsilon_sp^s+\sum\limits_{t\ge1}2r_tp^t\equiv
2\varepsilon_0 \text{ modulo } 2p.$$ 
Thus (3) follows from \fullref{Milnor's basis2}. 
\end{proof}

\section{More on excess filtration} \label{sec:sec2}

For $R=(r_1,r_2,\dots,r_n,\dots)\in\seq$, put
$s(R)=(0,r_1,r_2,\dots,r_n,\dots)$. 

If some entry of $R$ is not a non-negative integer, we put $\wp(R)=0$. 
We regard $\seq$ as a monoid with componentwise addition, then  
$\bzr=(0,0,\dots,0,\dots)$ is the unit of $\seq$. 
Let $E_n$ be an element of $\seq^b$ such that the $n$th entry is 1 and 
other 
entries are all $0$. 
(We put $E_0=\bzr$.) 

\begin{lem}\label{cong1} \qua
\begin{enumerate}
\item[\rm(1)] If $\varepsilon=0,1$, $|E|+2|R|\le 2i-j+1$ and $Q(E)\wp(R)\in 
\ca_p^j$, 
$$\beta^{\varepsilon}\wp^iQ(E)\wp(R)\equiv Q(\varepsilon E_1+s(E))
\wp\bigl(\bigl(i-\tfrac12(|E|+j)-|R|\bigr)E_1+s(R)\bigr)$$
modulo $F_{2i-j+\varepsilon+1}\ca_p$ for an odd prime $p$. 
\item[\rm(2)] If $|R|\le i-j$ and $\Sq(R)\in \ca_2^j$, 
$$\Sq^i\Sq(R)\equiv\Sq((i-j-|R|)E_1+s(R))\text{ modulo }F_{i-j+1}\ca_2.$$ 
\end{enumerate}
\end{lem}

\begin{proof}
Let $E=(\varepsilon_0,\varepsilon_1,\varepsilon_2,\dots)\in\seq^b$ and 
$R=(r_1,r_2,\dots)\in\seq$. 
We put $Q(E)=Q_{n_1}Q_{n_2}\cdots Q_{n_k}$ ($0\le n_1<n_2<\cdots<n_k$). 
By Milnor \cite[Theorem 4a, 4b]{Mi}, we have 
$$\beta^{\epsilon}\wp^iQ_{n_1}Q_{n_2}\cdots Q_{n_k}
=\sum\limits_{e_t=0,1}\beta^{\epsilon}Q_{n_1+e_1}Q_{n_2+e_2}
\cdots\beta^{\epsilon}Q_{n_k+e_k}\wp^{i-\sum_{t=1}^ke_tp^{n_t}}$$
and
$$\wp^m\wp(R)={\sum\limits_{\sum_{s\ge0}a_sp^s=m}}\hspace{5pt}{\prod_{s\ge1}}
\binom{r_s-a_s+a_{s-1}}{a_{s-1}}\wp(r_1-a_1+a_0,r_2-a_2+a_1,\dots).$$

Thus $\beta^{\epsilon}\wp^iQ_{n_1}Q_{n_2}\cdots Q_{n_k}\wp(R)$ is a linear 
combination of the Milnor basis 
$$Q_0^{\epsilon}Q_{n_1+e_1}Q_{n_2+e_2}\cdots 
Q_{n_k+e_k}\wp(r_1-a_1+a_0,r_2-a_2+a_1,\dots)$$
for $e_1,e_2,\dots,e_k=0,1$ and non-negative integers $a_0,a_1,a_2,\dots$  
satisfying 
$$a_0=i-\sum\limits_{t=1}^ke_tp^{n_t}-\sum\limits_{s\ge1}a_sp^s
\text{ and } a_s\le r_s \text{ for } s=1,2,\dots.$$ 

Suppose that sequences of non-negative integers $e_1,e_2,\dots,e_k$ and 
$a_0,a_1,a_2,\dots$ satisfy $e_j=0$ or $1$, $a_s\le r_s$ and 
$a_0=i-\sum_{t=1}^ke_tp^{n_t}-\sum_{s\ge1}a_sp^s$. 
We note that 
\begin{equation}\label{ineq1}
a_0\ge i-\sum\limits_{t=1}^kp^{n_t}-\sum\limits_{s\ge1}r_sp^s.
\end{equation}

Let $F$ be a sequence of integers such that 
$Q^F=Q_0^{\epsilon}Q_{n_1+e_1}Q_{n_2+e_2}\cdots Q_{n_k+e_k}$ and put 
$S=(r_1-a_1+a_0,r_2-a_2+a_1,\dots)$. 

Assume that $|E|+2|R|\le2i-j+1$ and $Q(E)\wp(R)\in \ca_p^j$. 
Since 
$$j=\sum\limits_{s\ge0}\varepsilon_s(2p^s-1)+\sum\limits_{t\ge1}2r_t(p^t-1)
=2\sum\limits_{s\ge0}\varepsilon_sp^s+2\sum\limits_{t\ge1}r_tp^t-|E|-2|R|,$$ 
we have 
$$\sum\limits_{t=1}^kp^{n_t}+\sum\limits_{t\ge1}r_tp^t=
\sum\limits_{s\ge0}\varepsilon_sp^s+\sum\limits_{t\ge1}r_tp^t
=\tfrac12(|E|+2|R|+j)\le i+\tfrac12.$$
Hence the right hand side of \eqref{ineq1} is non-negative and $a_0$ 
takes the 
minimum value 
$$i-\sum\limits_{t=1}^kp^{n_t}-\sum\limits_{s\ge1}r_sp^s=
i-\tfrac12(|E|+j)-|R|$$
if and only if $a_s=r_s$ for $s=1,2,\dots$ and $e_1=e_2=\dots=e_k=1$. 
In this case, 
$F=(\varepsilon,\varepsilon_0,\varepsilon_1,\varepsilon_2,\dots)
$ and $S=\left(i-\frac12(|E|+j)-|R|,r_1,r_2,\dots\right)$. 
Therefore 
$$|F|+2|S|=|E|+2|R|+\varepsilon+2\left(i-\tfrac12(|E|+j)-|R|\right)
=2i-j+\varepsilon$$
and the result follows. 
Proof for the case $p=2$ is similar. 
\end{proof}

Put $E_i^*\ca_p=\sum_{j\in\bz}E_i^j\ca_p$. 
Since the excess filtration $\fF_p$ satisfies (E3) and (E4) of 
\fullref{fil 3}, $E_i^*\ca_p$ is a left $\ca_p$--module and the product map 
$\mu \co \ca_p\otimes\ca_p\to\ca_p$ induces the following maps. 
$$\mu_i \co \ca_p\otimes E_i^*\ca_p\to E_i^*\ca_p, \qua 
{\tilde\mu}_i^{k,j} \co E_i^k\ca_p\otimes(\ca_p/F_{i-j-1}\ca_p)^j\to 
E_{i-j}^{k+j}\ca_p.$$ 

\begin{thm}\label{a5}
For non-negative integer $i$, $j$ and $\varepsilon=0,1$, the following map 
is 
an isomorphism.
$$\tilde\mu_{2i+\varepsilon}^{2i(p-1)+\varepsilon,j} \co
E_{2i+\varepsilon}^{2i(p-1)+\varepsilon}\ca_p\otimes(\ca_p/F_{2i-j+
\varepsilon
+1}\ca_p)^j\to E_{2i-j+\varepsilon}^{2i(p-1)+j+\varepsilon}\ca_p.$$
\end{thm}

\begin{proof}
Suppose $Q^F\wp(S)\in \ca_p^{2i(p-1)+j+\varepsilon}$ and 
$|F|+2|S|=2i-j+\varepsilon$ for 
$F=(\lambda_0,\lambda_1,\lambda_2,\dots)$ and $S=(s_1,s_2.\dots)$. 
Then, 
\begin{eqnarray}
\sum\limits_{k\ge0}\lambda_k+2\sum\limits_{k\ge1}s_k&=&2i-j+\varepsilon
\label{eqn:eqn2}
\\
\sum\limits_{k\ge0}\lambda_k(2p^k-1)+2\sum\limits_{k\ge1}s_k(p^k-1)&=&
2i(p-1)+j+\varepsilon.
\label{eqn:eqn3}
\end{eqnarray}
Hence 
\begin{equation}
\sum\limits_{k\ge0}\lambda_kp^k+\sum\limits_{k\ge1}s_kp^k=ip+\varepsilon,
\label{eqn:eqn4}
\end{equation}
and this implies $\lambda_0=\varepsilon$. 
We put $E=(\lambda_1,\lambda_2,\dots)$ and $R=(s_2,s_3.\dots)$. 
By \eqref{eqn:eqn2} above, we have $s_1=i-\frac12(|E|+j)-|R|$. 
Therefore, $\beta^{\varepsilon}\wp^iQ(E)\wp(R)\equiv Q^F\wp(S)$ modulo 
$F_{2i-j+\varepsilon+1}\ca_p$ by \fullref{cong1}. 
This shows that $\smash{\tilde\mu_{2i+\varepsilon}^{2i(p-1)+\varepsilon,j}}$ is 
surjective. 
It is clear from \fullref{cong1} that 
$\smash{\tilde\mu_{2i+\varepsilon}^{2i(p-1)+\varepsilon,j}}$ is injective. 
The proof for the case $p=2$ is similar. 
\end{proof}

For $R=(r_1,r_2,\dots,r_n,\dots)\in\seq$ and a non-zero integer $p$, we 
say that $p$ divides $R$ if $p|r_i$ for all $i\ge1$ and denote this by 
$p|R$ 
and by $p\ndiv R$ otherwise. 
Put 
$\tfrac1pR=\bigl(\tfrac{r_1}p,\tfrac{r_2}p,\dots,\tfrac{r_n}p,\dots\bigr)$
if $p|R$. 

\begin{lem}\label{cong2}
Let $p$ be an odd prime. 
For $E\in\seq^b$, $R\in\seq$ and $j\ge0$, the following congruences hold. 

\begin{enumerate}
\item[\rm(1)] If $|R|\le pj$ and $p|R$, 
$$\wp(R)\wp^j\equiv\wp\bigl(\bigl(j-\tfrac1p|R|\bigr)E_1+
s\bigl(\tfrac1pR\bigr)\bigr)\text{ modulo }F_{2j+1}\ca_p.$$ 
If $E\ne\bzr$ or $|R|>pj$ or $p\ndiv R$, $Q(E)\wp(R)\wp^j\in 
F_{2j+1}\ca_p$. 
\item[\rm(2)] If $|R|\le pj+1$ and $p|R$, 
$$\wp(R)\beta\wp^j\equiv\beta\wp\bigl(\bigl(j-\tfrac1p|R|\bigr)E_1+
s\bigl(\tfrac1pR\bigr)\bigr)\text{ modulo }F_{2j+2}\ca_p.$$ 
If $|R|\le pj+1$ and $p|R-E_n$ for some $n\ge1$, 
$$\wp(R)\beta\wp^j \equiv Q_n\wp\bigl(\bigl(j-\tfrac1p(|R|-1)\bigr)E_1+
s\bigl(\tfrac1p(R-E_n)\bigr)\bigr)
 \text{ modulo }F_{2j+2}\ca_p.$$
If $E\ne\bzr$ or $|R|>pj+1$ or $p\ndiv R-E_n$ for any $n\ge0$, 
$Q(E)\wp(R)\beta\wp^j\in F_{2j+2}\ca_p$. 
\end{enumerate}
\end{lem}

\begin{proof}
(1)\qua By Milnor \cite[Theorem 4b]{Mi}, we have 
\begin{align*} 
\wp(R)\wp^j& =\sum_{x_0+x_1+\cdots=j}\prod_{k\ge0}\binom{r_k-px_k+x_{k-1}}
{x_{k-1}} \\
& \qua \wp(r_1-px_1+x_0,r_2-px_2+x_1,\dots),
\end{align*}
for $R=(r_1,r_2,\dots)$. 
Since $\binom{r_k-px_k+x_{k-1}}{x_{k-1}}=0$ if $r_k<px_k$, the summation 
of 
the right hand side of the above is taken over non-negative integers 
$x_0,x_1,\dots$ satisfying $x_0+x_1+\cdots=j$ and $px_k\le r_k$ for all 
$k=1,2,\dots$. 

Hence $p(j-x_0)=p(x_1+x_2+\cdots)\le|R|$ and $p(j-x_0)=|R|$ holds if and 
only if $px_k=r_k$ for all $k=1,2,\dots$. 

Put
$$S=(r_1-px_1+x_0,\dots,r_k-px_k+x_{k-1},\dots),$$
then 
$|S|=|R|-p(j-x_0)+j\ge j$ and $|S|=j$ hold if and only if $p|R$, $|R|\le 
pj$ and $S=\bigl(j-\tfrac1p|R|\bigr)E_1+s\bigl(\tfrac1pR\bigr)$. 

Therefore $Q(E)\wp(R)\wp^j\in F_{2j+1}\ca_p$ unless $E=\bzr$, $|R|\le pj$ 
and 
$p|R$. 

(2)\qua Since $\wp(R)\beta=\sum_{n\ge0}Q_n\wp(R-E_n)$ by  
Milnor \cite[Theorem 4a]{Mi}, the result follows from (1).
\end{proof}

In the case $p=2$, a similar result holds. 

\begin{lem}\label{cong22}
For $R\in\seq$ and $j\ge0$, the following congruences hold. 

If $|R|\le j$ and $2|R$, 
$$\Sq(R)\Sq^j\equiv\Sq\bigl(\bigl(j-\tfrac12|R|\bigr)E_1+
s\bigl(\tfrac12R\bigr)\bigr)\text{ modulo }F_{j+1}\ca_2.$$ 
If $|R|>j$ or $2\ndiv R$, $\Sq(R)\Sq^j\in F_{j+1}\ca_2$. 
\end{lem}

\begin{lem}\label{cong3}
Let $p$ be an odd prime, $R\in\seq$ and $j\ge0$. 
If $\wp(R)\in\ca_p^k$. then, the following congruences hold. 
\begin{enumerate}
\item If $|R|\le pj$ and $p|R$, 
$$\wp(R)\wp^j\equiv\wp^{j+\frac k{2p}}
\wp\bigl(\tfrac1pR\bigr)\text{ modulo }F_{2j+1}\ca_p.$$
\item If $|R|\le pj+1$ and $p|R$, 
$$\wp(R)\beta\wp^j\equiv\beta\wp^{j+\frac k{2p}}
\wp\bigl(\tfrac1pR\bigr)\text{ modulo }F_{2j+2}\ca_p,$$
if $|R|\le pj+1$ and $p|R-E_n$ for some $n\ge1$, 
$$\wp(R)\beta\wp^j\equiv\wp^{j+\frac{k+2}{2p}}
Q_{n-1}\wp\bigl(\tfrac1p(R-E_n)\bigr)\text{ modulo }F_{2j+2}\ca_p.$$
\end{enumerate}
\end{lem}

\begin{proof}
The first congruence and is a direct consequence of \fullref{cong1} and 
\fullref{cong2}. Suppose $p|R-E_n$ and $|R|\le pj+1$. 
By Milnor \cite[Theorem 4a]{Mi}, \fullref{cong1} and \fullref{cong2}, 
\begin{align*}
\wp^{j+\frac{k+2}{2p}}&Q_{n-1}\wp\bigl(\tfrac1p(R-E_n)\bigr)\\
&=Q_{n-1}\wp^{j+\frac{k+2}{2p}}\wp\bigl(\tfrac1p(R{-}E_n)\bigr)+
Q_n\wp^{j+\frac{k+2}{2p}-p^{n-1}}\wp\bigl(\tfrac1p(R-E_n)\bigr)\\
&\equiv Q_n\wp\bigl(\bigl(j-\tfrac1p(|R|-1)\bigr)E_1+s(\tfrac1pR)\bigr)
\text{ modulo }F_{2j+2}\ca_p\\
&\equiv\wp(R)\beta\wp^j\text{ modulo }F_{2j+2}\ca_p. 
\end{align*}
We also obtain $\wp(R)\beta\wp^j\equiv\beta\wp^{j+\frac k{2p}}
\wp\bigl(\frac1pR\bigr)$ if $|R|\le pj+1$ and $p|R$ from \fullref{cong1} 
and 
\fullref{cong2}. 
\end{proof}

\begin{lem}\label{cong32}
For $R\in\seq$ and $j\ge0$, if $|R|\le j$, $2|R$ and $\Sq(R)\in\ca_2^k$, 
$$\Sq(R)\Sq^j\equiv\Sq^{j+\frac k2}
\Sq\bigl(\tfrac12R\bigr)\text{ modulo }F_{j+1}\ca_2.$$
\end{lem}

For non-negative integers $i$, $j$ and $\varepsilon=0,1$, put 
$\kappa=\varepsilon$ if $j$ is even and $\kappa=1-\varepsilon$ if $j$ is 
odd. 
Let $\gamma_{i,j,\varepsilon}$ be the composition of maps 
$$\mu_{2i-j+\varepsilon} \co \ca_p^{pj-(p-2)(\varepsilon-\kappa)}\otimes 
E_{2i-j+\varepsilon}^{(2i-j+\varepsilon-\kappa)(p-1)+\kappa}\ca_p\to 
E_{2i-j+\varepsilon}^{2i(p-1)+j+\varepsilon}\ca_p$$
and 
$$(\tilde\mu_{2i+\varepsilon}^{2i(p-1)+\varepsilon,j})^{-1} \co
E_{2i-j+\varepsilon}^{2i(p-1)+j+\varepsilon}\ca_p\to 
E_{2i+\varepsilon}^{2i(p-1)+\varepsilon}\ca_p\otimes
(\ca_p/F_{2i-j+\varepsilon+1}\ca_p)^j.$$

Let us denote by $\rho \co \ca_p^{pj}\to\ca_p^j$ the $p$th root map, that is, 
the dual of $p$th power map $\ca_{p*}^j\to\ca_{p*}^{pj}$, $x\mapsto x^p$. 
By Milnor \cite[Lemma 9]{Mi}, we have 
$$\rho(Q(E)\wp(R))=\begin{cases}\wp\bigl(\frac1pR\bigr)&E=\bzr,\;p|R\\
0&\text{otherwise}.\end{cases}$$
Let $\pi_i \co\ca_p\to\ca_p/F_i\ca_p$ be the quotient map. 
Put $g_{2i+\varepsilon}=\pi_{2i+\varepsilon+1}(\beta^{\varepsilon}\wp^i)$, 
then, $g_{2i+\varepsilon}$ generates 
$E_{2i+\varepsilon}^{2i(p-1)+\varepsilon}
\ca_p$. 
The next result is a direct consequence of \fullref{cong2} and 
\fullref{cong3}. 

\begin{prop}\label{cong4}
Let $i$, $j$, $k$ be non-negative integers, $\varepsilon=0,1$ and $p$ an 
odd 
prime. 
\begin{enumerate}
\item[\rm(1)] $\gamma_{i+j,2j,\varepsilon}\co \ca_p^{2jp}\otimes E_{2i+\varepsilon}
^{2i(p-1)+\varepsilon}\ca_p\to 
E_{2(i+j)+\varepsilon}^{2(i+j)(p-1)+\varepsilon}
\ca_p\otimes(\ca_p/F_{2i+1}\ca_p)^{2j}$ \linebreak maps $\theta\otimes 
g_{2i+\varepsilon}
\in\ca_p^{2jp}\otimes E_{2i+\varepsilon}^{2i(p-1)+\varepsilon}\ca_p$ to 
$g_{2i+2j+\varepsilon}\otimes\pi_{2i+1}\rho(\theta)$. 
\item[\rm(2)] $\gamma_{i+j,2j+1,1} \!\co\! \ca_p^{2jp+2} \! \otimes\! 
E_{2i}^{2i(p-1)} \! \ca_p \! \to \! E_{2(i+j)+1}^{2(i+j)(p-1)+1}
\! \ca_p\!\otimes\!(\ca_p/F_{2i+1}\ca_p)^{2j+1}$ is a 
trivial map. 
\item[\rm(3)] $\gamma_{i+j,2j-1,0} \!\co \! \ca_p^{2jp-2} \! \otimes \!
E_{2i+1}^{2i(p-1)+1} \! \ca_p  \!
\to \!
E_{2(i+j)}^{2(i+j)(p-1)} \! \ca_p \! \otimes (\ca_p/F_{2i+2}\ca_p)^{2j-1}$ 
maps $(F_{kp+2}\ca_p)^{2jp-2}\otimes E_{2i+1}^{2i(p-1)+1}\ca_p$
into 
$$E_{2(i+j)}^{2(i+j)(p-1)} \ca_p \otimes
(F_{k+1}\ca_p/F_{2i+2}\ca_p)^{2j-1}.$$ 
\end{enumerate}
\end{prop}

For $p=2$, we have the following Proposition. 

\begin{prop}\label{cong42}
Let $i$, $j$ be non-negative integers. 
$$\gamma_{i,j,\varepsilon} \co \ca_2^{2j}\otimes E_{2i-j+\varepsilon}
^{2i-j+\varepsilon}\ca_2\to E_{2i+\varepsilon}^{2i+\varepsilon}
\ca_2\otimes(\ca_2/F_{2i-j+\varepsilon+1}\ca_2)^j$$
maps $\theta\otimes g_{2i-j+\varepsilon}
\in\ca_2^{2j}\otimes E_{2i+\varepsilon}^{2i+\varepsilon}\ca_2$ to 
$g_{2i+\varepsilon}\otimes\pi_{2i-j+\varepsilon+1}\rho(\theta)$. 
\end{prop}

\section{Filtered Hopf algebra} \label{sec:sec3}

We denote by $\ce^*$ the category of graded vector spaces over a field $K$ 
and linear maps preserving degrees. 
We also denote by $\ce$ the category of (ungraded) vector spaces over $K$. 
For $n\in\bz$, define functors $\Sigma^n \co \ce^*\to\ce^*$, 
$\epsilon_n \co \ce^*\to\ce$ and $\iota_n \co \ce\to\ce^*$ as follows. 
$$(\Sigma^nV^*)^i=V^{i-n},\quad (\Sigma^nf)^i=f^{i-n},\quad
\epsilon_n(V^*)=V^n, \quad \epsilon_n(f)=f^n,$$
for an object $V^*$ and morphism $f$ of $\ce^*$. 
$$\iota_n(W)^k=\begin{cases}W& k=n\\0&k\ne n\end{cases}, \qquad
\iota_n(g)^k=\begin{cases} g&k=n\\0&k\ne n,\end{cases}$$
for an object $W$ and morphism $g$ of $\ce$. 

\begin{prop}\label{preserving limits}
$\iota_n$ is a right and left adjoint of $\epsilon_n$. 
\end{prop}

\begin{proof}
Define natural transformations $u_n \co \id_{\ce^*}\to\iota_n\epsilon_n$, 
${\bar u}_n \co \epsilon_n\iota_n\to \id_{\ce}$, 
${\bar c}_n \co \id_{\ce}\to\epsilon_n\iota_n$ and 
$c_n \co \iota_n\epsilon_n\to \id_{\ce^*}$ as follows. 
For $V^*\in\ob\,\ce^*$, 
$$u_{n\,V^*}(x)=\begin{cases} x&x\in V^n\\0&x\in V^k,k\ne 
n\end{cases},\qquad c_{n\,V^*}(x)=x \qua (x\in V^n).$$ 
For $U\in\ob\,\ce$, ${\bar u}_{n\,U}(y)=y$ ($y\in(\epsilon_n\iota_n(U))^n
=U$), ${\bar c}_{n\,U}(y)=y$ ($y\in U$). 
Clearly, $c_{n\,V^*} \co \iota_n\epsilon_n(V^*)\to V^*$ is an inclusion map 
and 
${\bar u}_{n\,U} \co \epsilon_n\iota_n(U)\to U$ and 
${\bar c}_{n\,U} \co U\to\epsilon_n\iota_n(U)$ can be regarded as identity 
maps. 
Then, $u_n$ and ${\bar u}_n$ are the unit and the counit of the adjunction 
$\epsilon_n\vdash\iota_n$ respectively, and ${\bar c}_n$ and $c_n$ are 
the unit and the counit of the adjunction $\iota_n\vdash\epsilon_n$ 
respectively. 
\end{proof}

Let $A^*$ be a graded Hopf algebra over $K$ with an decreasing 
filtration $\fF=(F_iA^*)_{i\in\bz}$ of subspaces of $A^*$. 
The notion of unstable $A^*$--module is defined as follows. 

\begin{defn}
A left $A^*$--module $M^*$ with structure map $\alpha \co A^*\otimes M^*\to 
M^*$ is 
called an unstable $A^*$--module with respect to $\fF$ if 
$\alpha(F_{n+1}A^*\otimes M^n)=\{0\}$ for $n\in\bz$. 
We denote by $\cu\cm(A^*)$ the full subcategory of the category of left 
$A^*$--modules consisting of unstable $A^*$--modules. 
\end{defn}

We are going to give conditions on $\fF$ which suffices to develop a
theory of unstable $A^*$--modules.  The following is the first one.

\begin{cond}\label{filt} \quad
\begin{enumerate}
\item {(E1)} $F_iA^*=A^*$ if $i\le0$. 
\item {(E2)} $\bigcap_{i\in\bz}F_iA^*=\{0\}$. 
\end{enumerate}
\end{cond}

Note that if $\fF$ satisfies (E1) and $V^*$ is an unstable $A^*$--module, 
$V^n=\{0\}$ for $n<0$. 
The next one comes from \fullref{fil 3}. 

\begin{cond}\label{alg filt}
Let us denote by $\mu \co A^*\otimes A^*\to A^*$ and $\delta:A^*\to A^*\otimes 
A^*
$ the product and the coproduct of $A^*$, respectively. 
For an decreasing filtration $\fF=(F_iA^*)_{i\in\bz}$ of subspaces of 
$A^*$, 
we consider the following conditions. 
\begin{enumerate}
\item {(E3)} $F_iA^*$'s are left ideals of $A^*$ for $i\in\bz$. 
\item {(E4)} $\mu(F_iA^*\otimes A^j)\subset F_{i-j}A^*$ for $i,j\in\bz$. 
\item {(E5)} $\delta(F_iA^*)\subset
\sum_{j+k=i}F_jA^*\otimes F_kA^*$ for $i\in\bz$. 
\end{enumerate}
\end{cond}

We remark that if $\fF$ satisfies (E3) of \fullref{alg filt} and 
$\Sigma^n(A^*/F_{n+1}A^*)$ is an unstable $A^*$--module, then $\fF$ 
satisfies 
(E4) of \fullref{alg filt}. 
It is easy to verify the following fact. 

\begin{prop}
Let $A^*$ be a graded Hopf algebra over $K$ with decreasing filtration 
$\fF$.  
Suppose that $\fF$ satisfies the condition (E5) in \fullref{alg filt}. 
If $V^*$ and $W^*$ are unstable $A^*$--modules with respect to $\fF$, then 
so is $V^*\otimes W^*$. 
\end{prop}

Let us denote by $\cco \co \cu\cm(A^*)\to\ce^*$ the forgetful functor. 
Suppose that $\fF$ satisfies (E3) and (E4) of \fullref{alg filt}. 
Define a functor $\cf \co \ce^*\to\cu\cm(A^*)$ by 
$$\cf(V^*)=\sum_{n\in\bz}A^*/F_{n+1}A^*\otimes V^n\qquad\text{and}\qquad
\cf(f)=\sum_{n\in\bz}\id_{A^*/F_{n+1}A^*}\otimes f^n.$$

For an object $M^*$ of $\cu\cm(A^*)$, let 
$\alpha_n \co A^*/F_{n+1}A^*\otimes M^n\to M^*$ ($n\in\bz$) be the maps 
induced by 
the structure map $\alpha \co A^*\otimes M^*\to M^*$. 
These maps induce $\varepsilon_{M^*} \co \cf\cco(M^*)\to M^*$. 

Let $1_n$ be the class of $1\in A^0$ in $A^*/F_{n+1}A^*$. 
For an object $V^*$ of $\ce^*$, define a map 
$\eta_{V^*} \co V^*\to\cco\cf(V^*)$ by 
$\eta_{V^*}(x)=\sum_{n\in\bz}1_n\otimes u_{n\,V^*}(x)$ for $x\in 
V^*$. 

\begin{prop}\label{fum}
$\cf$ is a left adjoint of $\cco$. 
\end{prop}

\begin{proof}
It can be easily verified that $\eta \co \id_{\ce}\to\cco\cf$ (resp. 
$\varepsilon \co
\cf\cco\to \id_{\cu\cm(A^*)}$) is the unit (resp. counit) of the adjunction 
$\cf\vdash\cco$. 
\end{proof}

\begin{rem}\label{bar} \qua
\begin{enumerate}
\item[(1)] As a special case of the above result, we see that 
$\cf(\Sigma^nK)=\Sigma^nA^*/F_{n+1}A^*$ represents a functor 
$\epsilon_n\cco:\cu\cm(A^*)\to\ce$. 
Thus we can verify the fact that a functor $G \co \cu\cm(A^*)^{op}\to\ce$ is 
representable if $G$ is right exact and preserves direct sums 
(Lannes--Zarati \cite{LZ}). 
\item[(2)] The above result implies that $\cu\cm(A^*)$ has enough projectives and 
we 
can construct the bar resolutions (MacLane \cite{Ma}) in $\cu\cm(A^*)$ and that, 
if 
$\cf$ also satisfies (E5) and $L^*$ is an unstable $A^*$--module of finite 
type, the left adjoint to the functor $M^*\mapsto M^*\otimes L^*$ exists. 
\end{enumerate}
\end{rem}

Put $E_i^jA^*=(F_iA^*)^j/(F_{i+1}A^*)^j$ and 
$E_i^*A^*=\sum_{j\in\bz}E_i^jA^*$. 
If $\fF$ satisfies (E3) of \fullref{alg filt}, $E_i^*A^*$ is a left 
$A^*$--module. 
If $\fF$ satisfies (E4) of \fullref{alg filt}, the product map 
$\mu \co A^*\otimes A^*\to A^*$ induces 
${\bar\mu}_i^{k,j} \co E_i^kA^*\otimes A^j\to E_{i-j}^{k+j}A^*$. 
Consider a bigraded vector space $E_*^*A^*=\sum_{i\in\bz}E_i^*A^*$. 
Then $E_*^*A^*$ has a structure of a right $A^*$--module given by 
${\bar\mu}_i^{k,j}$'s. 
Suppose $\fF$ satisfies both (E3) and (E4), then ${\bar\mu}_i^{k,j}$ 
induces ${\tilde\mu}_i^{k,j} \co E_i^kA^*\otimes(A^*/F_{i-j+1}A^*)^j\to 
E_{i-j}^{k+j}A^*$. 
We can regard ${\tilde\mu}_i^{*,j}$ as a map $E_i^*A^*\otimes
\iota_j\epsilon_j(A^*/F_{i-j+1}A^*)\to E_{i-j}^*A^*$ in $\ce^*$. 

\fullref{fil 1} and \fullref{a5} suggests the following conditions. 

\begin{cond}\label{top op}
Let $A^*$ be an algebra over a field $K$ of characteristic $p$ with an 
decreasing filtration $\fF=(F_iA^*)_{i\in\bz}$. 
\begin{enumerate}
\item {(E6)} $E_{2i+\varepsilon}^kA^*=\{0\}$ ($i,k\in\bz$, 
$\varepsilon=0,1$) 
holds if $k<2i(p-1)+\varepsilon$ or $2i+\varepsilon+k\not\equiv0,2$ modulo 
$2p$. 
\item {(E7)} $\dim E_{2i+\varepsilon}^{2i(p-1)+\varepsilon}A^*=1$ for 
$i\ge0$, 
$\varepsilon=0,1$. 
\item {(E8)} For non-negative integers $i$, $j$ and $\varepsilon=0,1$, the 
map
$$\tilde\mu_{2i+\varepsilon}^{2i(p-1)+\varepsilon,j} \co E_{2i+\varepsilon}
^{2i(p-1)+\varepsilon}A^*\otimes(A^*/F_{2i-j+\varepsilon+1}A^*)^j\to 
E_{2i-j+\varepsilon}^{2i(p-1)+j+\varepsilon}A^*$$
is an isomorphism.
\end{enumerate}
\end{cond}

\begin{rem}\label{top op2}
Since $E_j^{2i(p-1)+\varepsilon}A^*=\{0\}$ if $j>2i+\varepsilon$ by (E6), 
we have \linebreak $\dim(F_{2i+\varepsilon}A^*)^{2i(p-1)+\varepsilon}=1$ for $i\ge0$ 
and $\varepsilon=0,1$ by (E2) and (E7). 
We also have $(F_{2i+\varepsilon}A^*)^k=\{0\}$ if $k<2i(p-1)+\varepsilon$. 
\end{rem}

We assume that $\fF$ satisfies (E1), (E2), (E3), (E4), (E6), (E7) 
and (E8) for the rest of this section. 

\begin{prop}\label{um}
A left $A^*$--module $M^*$ with structure map $\alpha \co A^*\otimes M^*\to 
M^*$ is 
unstable if and only if 
$\alpha((F_{2i+\varepsilon}A^*)^{2i(p-1)+\varepsilon}\otimes M^k)=\{0\}$ 
for any $i\in\bz$, $\varepsilon=0,1$ such that $k<2i+\varepsilon$. 
\end{prop}

\begin{proof}
Suppose $\alpha((F_{2i+\varepsilon}A^*)^{2i(p-1)+\varepsilon}\otimes M^k)=
\{0\}$ for any $i\in\bz$, $\varepsilon=0,1$ and $k<2i+\varepsilon$. 
Since (E8) implies 
\begin{multline*}
\mu((F_{2i+\varepsilon}A^*)^{2i(p-1)+\varepsilon}\otimes A^j)+
(F_{2i-j+\varepsilon+1}A^*)^{2i(p-1)+j+\varepsilon}\\
= (F_{2i-j+\varepsilon}A^*)^{2i(p-1)+j+\varepsilon},
\end{multline*}
we have 
$$\alpha(\!(F_{2i-j+\varepsilon+1}A^*)^{2i(p-1)+j+\varepsilon}\otimes 
M^{k-j}){=}\alpha(\!(F_{2i-j+\varepsilon}A^*)^{2i(p-1)+j+\varepsilon}
\otimes M^{k-j})$$
if $k<2i+\varepsilon$.
By putting $n=k-j$, $s=2i-j+\varepsilon$ and $t=2i(p-1)+j+\varepsilon$, we 
see that 
\begin{equation}\label{fm}
\alpha((F_{s+1}A^*)^t\otimes M^n)=\alpha((F_sA^*)^t\otimes M^n)
\end{equation}
holds if $s>n$ and $s+t\equiv0,2$ modulo $2p$. 
Since $(F_{s+1}A^*)^t=(F_sA^*)^t$ by (E6), It follows from \eqref{fm} 
that 
$\alpha((F_{n+1}A^*)^t\otimes M^n)=\alpha((F_mA^*)^t\otimes M^n)$ for any 
$m>n$. 
Since $\alpha((F_mA^*)^t\otimes M^n)=\{0\}$ for sufficiently large $m$ by 
(E6) and (E2), we have $\alpha((F_{n+1}A^*)^t\otimes M^n)=\{0\}$. 

The converse follows from 
$$\alpha((F_{2i+\varepsilon}A^*)^{2i(p-1)+\varepsilon}\otimes M^k)
\subset\alpha(F_{k+1}A^*\otimes M^k)=\{0\}.\proved$$
\end{proof}

For non-negative integers $i$, $j$ and $\varepsilon=0,1$, put 
$\kappa=\varepsilon$ if $j$ is even and $\kappa=1-\varepsilon$ if $j$ is 
odd. 
Let $\gamma_{i,j,\varepsilon}$ be the composition of maps 
$$\mu_{2i-j+\varepsilon} \co A^{pj-(p-2)(\varepsilon-\kappa)}\otimes 
E_{2i-j+\varepsilon}^{(2i-j+\varepsilon-\kappa)(p-1)+\kappa}A^*\to 
E_{2i-j+\varepsilon}^{2i(p-1)+j+\varepsilon}A^*$$
and 
$$\bigl(\tilde\mu_{2i+\varepsilon}^{2i(p-1)+\varepsilon,j}\bigr)^{-1}
\co
E_{2i-j+\varepsilon}^{2i(p-1)+j+\varepsilon}A^*\to 
E_{2i+\varepsilon}^{2i(p-1)+\varepsilon}A^*\otimes
(A^*/F_{2i-j+\varepsilon+1}A^*)^j.$$

\begin{cond}\label{gamma}
For a real number $r$, let us denote by $\llbracket r \rrbracket$ the minimum integer 
among 
integers which are not less than $r$. 
\begin{enumerate}
\item{(E9)} $\gamma_{i,j,\varepsilon}$ maps 
$(F_kA^*)^{pj-(p-2)(\varepsilon-\kappa)}\otimes 
E_{2i-j+\varepsilon}^{(2i-j+\varepsilon-\kappa)(p-1)+\kappa}A^*$ into 
$$E_{2i+\varepsilon}^{2i(p-1)+\varepsilon}A^*\otimes\left(F_{\left\llbracket
\unfrac kp\right\rrbracket}A^*/F_{2i-j+\varepsilon+1}A^*\right)^j.$$ 
\end{enumerate}
\end{cond}

It follows from \fullref{cong4} and \fullref{cong42} that the excess 
filtration $\fF_p$ on $\ca_p$ satisfies the above condition. 

We can construct the functor $\Phi \co \cu\cm(A^*)\to\cu\cm(A^*)$ as in 
Li \cite{Li}. 
For an unstable $A^*$--module $M^*$, define an $A^*$--module $\Phi M^*$ as 
follows. 
Put 
$$\Phi M^*=\sum_{i\in\bz,\varepsilon=0,1}
E_{2i+\varepsilon}^{2i(p-1)+\varepsilon}A^*\otimes M^{2i+\varepsilon}.$$
In other words, 
$$(\Phi 
M^*)^k=\begin{cases}E_{2i+\varepsilon}^{2i(p-1)+\varepsilon}A^*\otimes 
M^{2i+\varepsilon}&k=2ip+2\varepsilon,\qua i\in\bz,\qua\varepsilon=0,1\\
\{0\}&k\not\equiv0,2\text{ modulo }2p.\end{cases}$$
We denote by $\mu_i \co A^*\otimes E_i^*A^*\to E_i^*A^*$ the map 
induced by the product $\mu$ of $A^*$. 
Note that $\mu_i \co A^j\otimes E_i^kA^*\to E_i^{j+k}A^*$ is trivial if 
$i+j+k\not\equiv0,2$ modulo $2p$. 
Let $\alpha_{M^*} \co A^*\otimes M^*\to M^*$ be the $A^*$--module structure 
map of $M^*$. 
Since $M^*$ is unstable, $\alpha_{M^*}$ induces 
$\bar\alpha_{M^*,i} \co A^*/F_{i-1}A^*\otimes M^i\to M^*$. 
We define $\alpha_{\Phi M^*} \co \otimes\Phi M^*\to\Phi M^*$ by the 
following 
compositions:
\begin{multline*}
A^{2jp}\otimes E_{2i+\varepsilon}^{2i(p-1)+\varepsilon}A^*\otimes 
M^{2i+\varepsilon}  \\
\xrightarrow{\gamma_{i+j,2j,\varepsilon}\otimes1} 
 E_{2(i+j)+\varepsilon}^{2(i+j)(p-1)+\varepsilon}A^*\otimes
(A^*/F_{2i+\varepsilon+1}A^*)^{2j}\otimes M^{2i+\varepsilon} \\
\xrightarrow{1\otimes\bar\alpha_{M^*,2i+\varepsilon}} 
E_{2(i+j)+\varepsilon}^{2(i+j)(p-1)+\varepsilon}A^*\otimes 
M^{2(i+j)+\varepsilon}; 
\end{multline*}
\begin{multline*}
A^{2jp+2}\otimes E_{2i}^{2i(p-1)}A^*\otimes M^{2i} \\
  \xrightarrow{\gamma_{i+j,2j+1,1}\otimes1} 
  E_{2(i+j)+1}^{2(i+j)(p-1)+1}A^*\otimes(A^*/F_{2i+1}A^*)^{2j+1}\otimes
  M^{2i} \\
\xrightarrow{1\otimes\bar\alpha_{M^*,2i}} 
  E_{2(i+j)+1}^{2(i+j)(p-1)+1}A^*\otimes M^{2(i+j)+1};
\end{multline*}
and
\begin{multline*}
A^{2jp-2}\otimes E_{2i+1}^{2i(p-1)+1}A^*\otimes M^{2i+1} \\
\xrightarrow{\gamma_{i+j,2j-1,0}\otimes1} 
  E_{2(i+j)}^{2(i+j)(p-1)}A^*\otimes(A^*/F_{2i+2}A^*)^{2j-1}\otimes 
  M^{2i+1} \\
\xrightarrow{1\otimes\bar\alpha_{M^*,2i+1}} 
  E_{2(i+j)}^{2(i+j)(p-1)}A^*\otimes M^{2(i+j)}.
\end{multline*}
Since $\mu(F_{2ip+2\varepsilon+1}A^*\otimes(F_{2i+\varepsilon}A^*)
^{2i(p-1)+\varepsilon})\subset F_{2i+\varepsilon+1}A^*$ for 
$\varepsilon=0,1$ 
and $i\in\bz$ by (E4), we deduce that $\Phi M^*$ is an unstable 
$A^*$--module. 

For a homomorphism $f \co M^*\to N^*$ between unstable $A^*$--modules, let  
$\Phi f \co \Phi M^*\to\Phi N^*$ be the map induced by 
$\id_{E_{2i+\varepsilon}^{2i(p-1)+\varepsilon}}A^*\otimes f$. 

Then $\Phi f$ is a homomorphism of left $A^*$--modules and $\Phi$ is an 
endofunctor of $\cu\cm(A^*)$. 
Let
$$\lambda_{M^*}^{2ip+2\varepsilon} \co (\Phi M^*)^{2ip+2\varepsilon}=
E_{2i+\varepsilon}^{2i(p-1)+\varepsilon}A^*\otimes M^{2i+\varepsilon}\to 
M^{2ip+2\varepsilon}\quad (i\in\bz, \varepsilon=0,1)$$
be the restriction of 
$\bar\alpha_{M^*,2i+\varepsilon} \co A^*/F_{2i+\varepsilon+1}A^*\otimes 
M^{2i+\varepsilon}\to M^*$. 
Thus we have a map $\lambda_{M^*} \co \Phi M^*\to M^*$. 
It is easy to verify that $\lambda_{M^*}$ is a homomorphism of left 
$A^*$--modules and we have a natural transformation 
$\lambda \co \Phi\to \id_{\cu\cm(A^*)}$.  

For an object $V^*$ of $\ce^*$, let $\rho_{V^*} \co \cf(V^*)\to\Sigma\cf
(\Sigma^{-1}V^*)$ be the map induced by the quotient map 
$A^*/F_{n+1}A^*\to 
A^*/F_nA^*$ and the identity maps $V^n\to V^n=(\Sigma^{-1}V^*)^{n-1}$. 

\begin{prop}\label{SES}
The following is a short exact sequence. 
$$\CD 0@>>>\Phi\cf(V^*)@>{\lambda_{\cf(V^*)}}>>\cf(V^*)@>{\rho_{V^*}}>>
\Sigma\cf(\Sigma^{-1}V^*)@>>>0\endCD$$
\end{prop}

\begin{proof}
By (E6) and (E8), $\lambda_{\cf(V^*)}$ is an injection onto 
$\sum_{n\in\bz}F_nA^*/F_{n+1}A^*\otimes V^n$, which is the kernel 
of 
$\rho_{V^*}$. 
\end{proof}

\begin{lem}\label{ker coker}
Let $M^*$ be an unstable $A^*$--module. 
\begin{enumerate}
\item $\Sigma^{-1}\Cok \lambda_{M^*}$ is an unstable $A^*$--module. 
\item If $\fF$ satisfies (E9) in \fullref{gamma}, 
$\Sigma^{-1}\Ker \lambda_{M^*}$ 
is an unstable $A^*$--module. 
\end{enumerate}
\end{lem}

\begin{proof}
(1)\qua Since $(\Im\lambda_{M^*})^{2ip+\varepsilon}=(F_{2i+\varepsilon}A^*)
^{2i(p-1)+\varepsilon}M^{2i+\varepsilon}$, we have 
$$(F_{2i+\varepsilon}A^*)^{2i(p-1)+\varepsilon}(\Cok\lambda_{M^*})
^{2i+\varepsilon}=\{0\}.$$ 
If $k<2i+\varepsilon$, instability of $M^*$ and \fullref{um} imply 
$$(F_{2i+\varepsilon}A^*)^{2i(p-1)+\varepsilon}(\Cok\lambda_{M^*})^k=\{0\}.
$$ 
Thus the assertion follows from \fullref{um}. 

(2)\qua Put $N^{2i+\varepsilon}=\{x\in M^{2i+\varepsilon}|\,
(F_{2i+\varepsilon}A^*)^{2i(p-1)+\varepsilon}x=\{0\}\}$. 
Then we have \linebreak $(\Ker\lambda_{M^*})^{2ip+2\varepsilon}=
E_{2i+\varepsilon}^{2i(p-1)+\varepsilon}A^*\otimes N^{2i+\varepsilon}$ and 
$(\!F_{2i+\varepsilon}\!A^*\!)^{2i(p-1)+\varepsilon}N^{2i+\varepsilon}=\{0\}$. 
By \fullref{um}, it suffices to show 
$$(F_{2j+\varepsilon'}A^*)^{2j(p-1)+\varepsilon'}
(E_{2i+\varepsilon}^{2i(p-1)+\varepsilon}A^*\otimes 
N^{2i+\varepsilon})=\{0\}$$ for non-negative integers $i$, $j$ and 
$\varepsilon,\varepsilon'=0,1$ 
satisfying $2j+\varepsilon'\ge2ip+2\varepsilon$. 
We may assume $2j(p-1)+\varepsilon'\equiv0,\pm2$ modulo $2p$, that is, 
$\varepsilon'=0$ and $j\equiv0,\pm1$ modulo $p$ for dimensional reason. 
If $j=kp$, then $k\ge i+\varepsilon$ and it follows from \fullref{gamma} 
that 
\begin{multline*}
(F_{2kp}A^*)^{2kp(p-1)}
(E_{2i+\varepsilon}^{2i(p-1)+\varepsilon}A^*\otimes N^{2i+\varepsilon}) \\
=E_{2(i+k(p-1))+\varepsilon}^{2(i+k(p-1))(p-1)+\varepsilon}A^*\otimes 
((F_{2k}A^*)^{2k(p-1)}N^{2i+\varepsilon})=\{0\}.
\end{multline*}
If $j=kp-1$, then $k\ge i+1$ and we only have to consider the case 
$\varepsilon=0$ for dimensional reason. 
Since $(F_{2k}A^*)^{2k(p-1)-1}=\{0\}$ by (E6) and (E2), we have 
\begin{multline*}
(F_{2kp-2}A^*)^{2p(kp-k-1)+2}(E_{2i}^{2i(p-1)}A^*\otimes N^{2i}) \\
=E_{2(i+kp-k-1)+1}^{2(i+kp-k-1)(p-1)+1}A^*\otimes
((F_{2k}A^*)^{2k(p-1)-1}N^{2i})=\{0\}.
\end{multline*}
If $j=kp+1$, then $k\ge i$ and we only have to consider the case 
$\varepsilon=1$ for dimensional reason. 
Again, using \fullref{gamma} and the instability of $M^*$, we see 
\begin{multline*}
(F_{2kp+2}A^*)^{2p(kp-k+1)-2}(E_{2i+1}^{2i(p-1)+1}A^*\otimes N^{2i+1}) \\
=E_{2(i+kp-k+1)+1}^{2(i+kp-k+1)(p-1)+1}A^*\otimes 
((F_{2k+1}A^*)^{2k(p-1)+1}N^{2i+1})=\{0\}.
\end{multline*}
This completes the proof.
\end{proof}

Define functors $\Omega,\Omega^1 \co \cu\cm(A^*)\to\cu\cm(A^*)$ by 
$\Omega(M^*)=\Sigma^{-1}\Cok\lambda_{M^*}$ and 
$\Omega^1(M^*)=\Sigma^{-1}\Ker\lambda_{M^*}$. 
Let us denote by $\tilde\eta_{M^*} \co M^*\to\Cok\lambda_{M^*}=\Sigma\Omega 
M^*$ 
the quotient map and by $\iota_{M^*} \co \Sigma\Omega^1M^*\to\Phi M^*$ the 
inclusion map. 
For a morphism $f \co M^*\to N^*$ of unstable modules, let 
$\Omega f \co \Omega M^*
\to\Omega N^*$ and $\Omega^1f \co \Omega^1M^*\to\Omega^1N^*$ be the unique 
maps 
that make the following diagram commute. 
$$\CD 
0@>>>\Sigma\Omega^1M^*@>{\iota_{M^*}}>>\Phi M^*@>{\lambda_{M^*}}>>M^*
@>{\tilde\eta_{M^*}}>>\Sigma\Omega M^*@>>>0\\
@.@VV{\Sigma\Omega^1f}V@VV{\Phi f}V@VV{f}V@VV{\Sigma\Omega f}V@.\\
0@>>>\Sigma\Omega^1N^*@>{\iota_{N^*}}>>\Phi N^*@>{\lambda_{N^*}}>>N^*
@>{\tilde\eta_{N^*}}>>\Sigma\Omega N^*@>>>0
\endCD$$

\begin{prop}\label{left adjoint}
$\Omega$ is the left adjoint of the suspension functor $\Sigma$. 
$\Omega^1$ is the first left derived functor of $\Omega$ and all the 
higher 
derived functors are trivial. 
\end{prop}

\begin{proof}
We first note that $\lambda_{\Sigma M^*}\co \Phi\Sigma M^*\to\Sigma M^*$ is 
trivial by the instability of $M^*$. 
Hence $\tilde\eta_{\Sigma M^*} \co \Sigma M^*\to\Sigma\Omega\Sigma M^*$ is an 
isomorphism. 
Define $\tilde\varepsilon_{M^*} \co \Omega\Sigma M^*\to M^*$ by 
$\tilde\varepsilon_{M^*}=\Sigma^{-1}\tilde\eta_{\Sigma M^*}^{-1}$. 
Obviously, 
$\Sigma\tilde\varepsilon_{M^*}\tilde\eta_{\Sigma M^*}=\id_{\Sigma M^*}$. 
By the naturality of $\lambda$ and the definition of $\tilde\varepsilon$, 
we 
have 
$$\Sigma(\tilde\varepsilon_{\Omega 
M^*}\Omega\tilde\eta_{M^*})\tilde\eta_{M^*}=
\Sigma\tilde\varepsilon_{\Omega M^*}(\Sigma\Omega\tilde\eta_{M^*})
\tilde\eta_{M^*}=\tilde\eta_{\Sigma\Omega M^*}^{-1}
\tilde\eta_{\Sigma\Omega M^*}\tilde\eta_{M^*}=\tilde\eta_{M^*}.$$ 
Hence $\tilde\varepsilon_{\Omega M^*}\Omega\tilde\eta_{M^*}=\id_{\Omega 
M^*}$ and $\Omega$ is the left adjoint of $\Sigma$. 

Let
$$M^*\xleftarrow{\varepsilon_{M^*}}B_0^*\xleftarrow{\partial_1}\cdots
\xleftarrow{\partial_{n-1}}B_{n-1}^*\xleftarrow{\partial_n}B_n^*
\xleftarrow{\partial_{n+1}}\cdots$$
be the bar resolution of $M^*$. 
Consider chain complexes
$$B.=(B_n^*,\partial_n)_{n\in\bz},\quad 
\Phi B.=(\Phi B_n^*,\Phi(\partial_n))_{n\in\bz} \text{ and }
\Sigma\Omega B.=(\Sigma\Omega B_n^*,\Sigma\Omega(\partial_n))_{n\in\bz}.$$ 
We denote by $\lambda. \co \Phi B.\to B.$ and 
$\eta. \co B.\to\Sigma\Omega B.$ the 
chain maps given by the $\lambda_{B_n^*}$ and $\eta_{B_n^*}$, 
respectively. 
Since
$$0\to\Phi B_n^*\xrightarrow{\lambda_{B_n^*}}B_n^*
\xrightarrow{\eta_{B_n^*}}\Sigma\Omega B_n^*\to0$$
is exact by \fullref{SES}. 
we have a short exact sequence of complexes 
$$0\to\Phi B.\xrightarrow{\lambda.}B.\xrightarrow{\eta.}\Sigma\Omega 
B.\to0.$$ 
Consider the long exact sequence associated with this short exact 
sequence. 
Clearly, $\Phi$ is an exact functor. 
We deduce that $\Sigma H^n(\Omega B.)=H^n(\Sigma\Omega B.)$ is trivial and 
that there is an exact sequence 
$$0\to\Sigma H^1(\Omega B.)=H^1(\Sigma\Omega B.)\to\Phi M_*
\xrightarrow{\lambda_{M^*}}M^*\xrightarrow{\eta_{M^*}}\Sigma\Omega 
M^*\to0.$$ 
Thus $\Omega^nM^*=H^n(\Omega B.)$ is trivial if $n>1$ and $\Omega^1$ 
defined 
above is the first left derived functor of $\Omega$. 
\end{proof}

\section{Unipotent group scheme} \label{sec:sec4}

For a commutative ring $k$, we denote by $\alg_k^*$ the category of graded 
$k$--algebras and by $h_{A^*}$ the functor represented by an object $A^*$ 
of $\alg_k$. 
We denote by $\gr$ the category of groups. 

For a Hopf algebra $A^*$, let us denote by $A_*$ the dual Hopf algebra, 
that 
is, $A_n$ is the dual vector space $\hom_K(A^n,K)$ and 
$A_*=\sum_{n\in\bz}A_n$. 
We assume that $A^*$ is finite type and that $A^n=0$ for $n<0$. 

For a filtration $\fF=(F_iA^*)_{i\in\bz}$ of $A^*$, define the dual 
filtration 
$\fF^*=(F_iA_*)_{i\in\bz}$ on $A_*$ by 
$$F_iA_n=\Ker\left(\kappa_{i+1} \co A_n=\hom_K(A^n,K)\to\hom_K(F_{i+1}A^n,K)
\right)
$$
Here, $\kappa_i \co F_iA^n\to A^n$ denotes the inclusion map. 
Note that the dual of the dual filtration $\fF^*$ is identified with 
$\fF$. 

We list conditions on the dual filtration. 

\begin{cond}\label{dual filt}
Let $\mu^* \co A_*\to A_*\otimes A_*$ (resp. $\delta^*:A_*\otimes A_*\to A_*$) 
be 
the coproduct (resp. product) of $A_*$. 
\begin{enumerate}
\item{(E$1^*$)} $F_iA_*=\{0\}$ if $i<0$. 
\item{(E$2^*$)} $\bigcup_{i\in\bz}F_iA_*=A_*$. 
\item{(E$3^*$)} $F_iA_*$'s are left coideals of $A_*$ (that is,
$\mu^*(F_iA_*)\subset A_*\otimes F_iA_*$) for $i\in\bz$. 
\item{(E$4^*$)} $\mu^*(F_iA_k)\subset\sum_{j\in\bz}
F_{j+i}A_{k-j}\otimes A_j$ for $i,j\in\bz$.
\item{(E$5^*$)} $\delta^*(F_jA_*\otimes F_kA_*)\subset F_{j+k}A_*$ for 
$j,k\in\bz$. 
\item{(E$6^*$)} $E_{2i+\varepsilon}^kA_*=\{0\}$ 
($i,k\in\bz$, $\varepsilon=0,1$) holds if $k<2i(p-1)+\varepsilon$ or 
$2i+\varepsilon+k\not\equiv0,2$ modulo $2p$. 
\item{(E$7^*$)} $\dim E_{2i+\varepsilon}^{2i(p-1)+\varepsilon}A^*=1$ for 
$i\ge0$, $\varepsilon=0,1$. 
\end{enumerate}
\end{cond}

It is easy to verify the following fact. 

\begin{prop}[Yamaguchi \cite{rs}]
\label{e1-7}
For $l=1,2,3,4,5,6,7$, $\fF$ satisfies the condition {\rm(E$l$)} if and only if 
$\fF^*$ satisfies {\rm(E$l^*$)}. 
\end{prop}

For a prime $p$, we define a graded Hopf algebra $A_{(p)*}$ over a prime 
field $\bff_p$ as follows. 
As an algebra, we put 
$$A_{(p)*}=E(x_{i1}|\,i\ge2)\otimes\bff_p[x_{ij}|\,i>j\ge2]\text{ if }
p\ne2,\qua A_{(2)*}=\bff_2[x_{ij}|\,i>j\ge1].$$
We assign the generators $x_{ij}$ degrees as follows. 
\begin{align*}
\deg\; x_{ij}&=\begin{cases}2p^{i-2}-1& i\ge2,j=1\\2p^{j-2}(p^{i-j}-1)&i>j
\ge2
\end{cases} & \text{if } p\ne2\\
\deg\; x_{ij}&= 2^{j-1}(2^{i-j}-1) & \text{if } p=2.
\end{align*}
Define the coproduct $\mu^*$ and the counit $\eta^*$ of $A_{(p)*}$ by 
$$\mu^*(x_{ij})=x_{ij}\otimes1+\sum_{k=j+1}^{i-1}x_{ik}\otimes x_{kj}
+1\otimes x_{ij},\qua \eta^*(x_{ij})=0.$$

Then, $A_{(p)*}$ is a commutative Hopf algebra and its conjugation 
(canonical anti-automorphism) $\iota^*$ is given by 
$$\iota^*(x_{ij})=-x_{ij}-\sum_{k=j+1}^{i-1}x_{ik}\iota^*(x_{kj}).$$

Hence the affine scheme $h_{A_{(p)*}}$ represented by $A_{(p)*}$ takes its 
values in the category of groups, namely, 
$h_{A_{(p)*}} \co \alg_{\bff_p}^*\to\gr$ 
is an affine group scheme. 

\begin{rem}\label{unipotent}
For a positive integer $n$ and a graded $\bff_p$--algebra $R^*$, let 
$U_n(R^*)$ 
be a set of $n\times n$ unipotent matrices $A$ whose $(i,j)$th entry 
$a_{ij}$ 
satisfies 
\begin{align*}
a_{ij}&\in
\begin{cases}
R^{2p^{i-2}-1}&i\ge2,j=1\\
R^{2p^{j-2}(p^{i-j}-1)}& i>j\ge2
\end{cases}
& \text{ if } p\ne2\\
a_{ij} & \in R^{2^{j-1}(2^{i-j}-1)} & \text{ if }p=2
\end{align*}
and $a_{11}=a_{22}=\cdots=a_{nn}=1$, $a_{ij}=0$ if $i<j$. 
Then, $U_n(R^*)$ is a group by the multiplication of matrices. 
Hence we have a $\bff_p$--group functor $U_n \co \alg_{\bff_p}^*\to\gr$. 
On the other hand, let $A(n)_{(p)*}$ be the Hopf subalgebra of $A_{(p)*}$ 
generated by $\{x_{ij}|\,1\le j<i\le n\}$. 
For a map $f \co A(n)_{(p)*}\to R^*$ of graded $K$--algebras, we denote by 
$A_f$ 
the element of $U_n(R^*)$ whose $(i,j)$th component is $f(x_{ij})$ if $i>j$. 
Define a map $\theta_{nR^*} \co h_{A(n)_{(p)*}}(R^*)\to U_n(R^*)$ by 
$\theta_{nR^*}(f)=A_f$. 
It is easy to verify that $\theta_{nR^*}$ is an isomorphism groups and we 
have 
a natural equivalence $\theta_n \co h_{A(n)_{(p)*}}\to U_n$. 

If $A=(a_{ij})\in U_{n+1}(R^*)$, let $A'$ be the $n\times n$ matrix whose 
$(i,j)$th component is $a_{ij}$. 
Then $A'\in U_n(R^*)$ and we define a morphism $\pi_n \co U_{n+1}\to U_n$ by 
$\pi_{nR^*}(A)=A'$. 
Let $U_{\infty}$ be the limit of the inverse system 
$$\bigl(U_{n+1}\!\xrightarrow{\pi_n}U_n\bigr)_{n=1,2,\dots}.$$ 
The morphism $\iota_n^* \co h_{A(n+1)_{(p)*}}{\to}h_{A(n)_{(p)*}}$ induced by 
the inclusion map
$\iota_n\co A(n)_{(p)*}\to\allowbreak A(n+1)_{(p)*}$
satisfies $\theta_n\iota_n^*=\pi_n\theta_{n+1}$. 
Since $A_{(p)*}$ is the colimit of the direct system 
$$\bigl(A(n)_{(p)*}\xrightarrow{\iota_n}A(n+1)_{(p)*}\bigr)_{n=1,2,\dots},$$
it follows that the $\theta_n$ induce a natural equivalence 
$\theta_{\infty} \co h_{A_{(p)*}}\to U_{\infty}$. 
Thus, $A_{(p)*}$ represents the group scheme of ``infinite dimensional 
unipotent matrices". 
\end{rem}

In order to relate $A_{(p)*}$ with the dual Steenrod algebra $\ca_{p*}$, 
we consider representation of an affine group scheme. 

\begin{defn}\label{module}
Let $V^*$ be a finite dimensional vector space over $K$. 
Define a functor $F_{V^*} \co \alg_K^*\to\ce^*$ by $F_{V^*}(R^*)=V^*\otimes 
R^*$. 
We regard $F_{V^*}(R^*)$ as a right $R^*$--module. 
\end{defn}

We denote by $V_n^*$ the graded vector space over $\bff_p$ such that 
$\dim V_n^k=1$ for $k=-1,-2,\dots,-2p^i,\dots,-2p^{n-2}$ and $V_n^k=\{0\}$ 
otherwise if $p\ne2$, $\dim V_n^k=1$ for 
$k=-1,-2,\dots,-2^i,\dots,-2^{n-1}$ 
and $V_n^k=\{0\}$ otherwise if $p=2$. 
Let $v_k$ be a base of $V_n^k$ for $k$ such that $\dim V_n^k=1$. 

Define $\alpha_{nR^*} \co F_{V_n^*}(R^*)\times U_n(R^*)\to F_{V_n^*}(R^*)$ by 
$$\alpha_{nR^*}(v_j\otimes1,(a_{ij}))=\sum_{i=1}^nv_i\otimes a_{ij}$$
so that $U_n(R^*)$ acts $R^*$--linearly on $F_{V_n^*}(R^*)$. 
Hence $F_{V_n^*}$ is a right $U_n$--module, in other words, 
$\alpha_n \co F_{V_n^*}\times U_n\to F_{V_n^*}$ is a representation of $U_n$ 
on $V_n^*$. 

Let $\varphi_n \co V_n^*\to V_n^*\otimes A(n)_{(p)*}$ be the map defined by 
$$\varphi_n(v_j)=\alpha_{nA(n)_{(p)*}}(v_j\otimes1,(x_{ij}))
=v_j\otimes1+\sum\limits_{i=j+1}^nv_i\otimes x_{ij}.$$
Here, we put $x_{jj}=1$ and $x_{ij}=0$ if $i<j$.
Then, $\varphi_n$ is a right comodule structure map of $V_n^*$. 
Composing the map $\id_{V_n^*}\otimes\kappa_n \co
V_n^*\otimes A(n)_{(p)*}\to V_n^*\otimes A_{(p)*}$ 
induced by the inclusion map $\kappa_n \co A(n)_{(p)*}\hookrightarrow 
A_{(p)*}$ 
to $\varphi_n$, $V_n^*$ is regarded as a right $A_{(p)*}$--comodule. 

We change the gradings of the mod $p$ cohomology group $H^*(X)$ of a space 
$X$ 
by replacing $H^n(X)$ by $H^{-n}(X)$ so that the Milnor coaction 
$\psi_X:H^*(X)\to H^*(X)\widehat\otimes\ca_{p*}$ preserves degrees. 
Recall that the Milnor coaction on the mod $p$ cohomology group of 
$B\bz/p\bz$ 
is a homomorphisms of algebras given as follows. 
$$\psi(t)=t\otimes1-\sum\limits_{k\ge0}s^{p^k}\otimes\tau_k \text{ and }
\psi(s)=\sum\limits_{k\ge0}s^{p^k}\otimes\xi_k\text{ if } p\ne2,$$
where $H^*(B\bz/p\bz)=E(t)\otimes\bff_p[s]$ 
($t\in H^{-1}(B\bz/p\bz)$, $s\in H^{-2}(B\bz/p\bz)$). 
$$\psi(t)=\sum\limits_{k\ge0}t^{2^k}\otimes\zeta_k\text{ if }p=2,$$
where $H^*(B\bz/p\bz)=\bff_2[t]$ ($t\in H^{-1}(B\bz/2\bz)$). 
We identify $V_n^*$ with a subspace of the mod $p$ cohomology group of the 
$(2p^{n-2}+1)$--skeleton (resp. $2^{n-1}$--skeleton) of $B\bz/p\bz$ spanned 
by $\bigl\{t,s,s^p,\dots,s^{p^{n-2}}\bigr\}$ 
(resp. $\bigl\{t,t^2,\dots,t^{2^{n-1}}\bigr\}$) if $p\ne2$
(resp. $p=2$).
Put $v_1=t$ and $v_j=s^{p^{j-2}}$ ($j=2,3,\dots,n$) if $p\ne2$ and 
$v_j=t^{2^{j-1}}$ ($j=1,\dots,n$) if $p=2$. 
By the above equality, we have 
$$\psi(v_1)=v_1\otimes1-\sum\limits_{i=2}^nv_i\otimes\tau_{i-2}, \qquad
\psi(v_j)=\begin{cases}
\sum\limits_{i=j}^nv_i\otimes\xi_{i-j}^{p^{j-2}}(2\le j\le n) &
  \text{if } p\ne2, \\
\sum\limits_{i=j}^nv_i\otimes\zeta_{i-j}^{2^{j-1}} &
  \text{ if }p=2.\end{cases}$$
Hence the map $\rho_p \co A_{(p)*}\to\ca_{p*}$ given by 
$\rho_p(x_{i1})=-\tau_{i-2}$, $\rho_p(x_{ij})=\xi_{i-j}^{p^{j-2}}$ 
($j\ge2$)
if $p\ne2$ and $\rho_2(x_{ij})=\zeta_{i-j}^{2^{j-1}}$ if $p=2$ is a map of 
Hopf algebras and the composition 
$$V_n^*\xrightarrow{\varphi_n}V_n^*\otimes A(n)_{(p)*}
\xrightarrow{\id\otimes\kappa_n}V_n^*\otimes A_{(p)*}
\xrightarrow{\id\otimes\rho_p}V_n^*\otimes\ca_{p*}$$
coincides with the Milnor coaction (See Yamaguchi \cite{rs} for details). 

\begin{rem}\label{imbedding}
Since $\rho_p(x_{s+2\,1})=-\tau_s$, $\rho_p(x_{s+2\,2})=\xi_s$ and 
$\rho_2(x_{s+1\,1})=\zeta_s$, $\rho_p$ is surjective. 
Hence the affine group scheme represented by $\ca_{p*}$ is regarded as a 
closed subgroup scheme of $U_{\infty}$. 

The kernel of $\rho_p$ is the ideal generated by $\bigl\{
x_{ij}-x_{i-j+2\,2}^{p^{j-2}}\big|\,i>j\ge3\bigr\}$ if $p\ne2$ and 
$\bigl\{x_{ij}-x_{i-j+1\,1}^{2^{j-1}}\big|\,i>j\ge2\bigr\}$ if 
$p=2$. 
\end{rem}

Let $F_iA_{(p)*}$ be the subspace of $A_{(p)*}$ spanned by 
$$\biggl\{x_{k_11}x_{k_21}\cdots x_{k_m1}x_{i_1j_1}x_{i_2j_2}\cdots 
x_{i_nj_n} \bigg|\,j_1,j_2,\dots,j_n\ge2,
m+2\sum\limits_{l=1}^np^{j_l-2}\le i\biggr\},$$ 
if $p\ne2$ and $F_iA_{(2)*}$ be the subspace of $A_{(2)*}$ spanned by 
$$\biggl\{x_{i_1j_1}x_{i_2j_2}\cdots x_{i_nj_n}\bigg|\,
\sum\limits_{l=1}^n2^{j_l-1}\le i\biggr\}.$$ 
By this definition and \fullref{dual filt3}, it is easy to verify the 
following assertions. 

\begin{prop}\label{induced filt} \qua
\begin{enumerate}
\item The filtration $\left(F_iA_{(p)*}\right)_{i\in\bz}$ on 
$A_{(p)*}$ satisfies the conditions (E$1^*$)$\sim$(E$6^*$). 
\item $\rho_p\left(F_iA_{(p)*}\right)=F_i\ca_{p*}$. 
\end{enumerate}
\end{prop}

It follows from \fullref{e1-7} that the dual filtration 
$\bigl(F_iA_{(p)}^*\bigr)_{i\in\bz}$ on the dual Hopf algebra $A_{(p)}^*$ 
of 
$A_{(p)*}$ satisfies the conditions (E1)$\sim$(E6). 
Note that the Steenrod algebra $\ca_p$ is a Hopf subalgebra of 
$A_{(p)}^*$. 

However, $\bigl(F_iA_{(p)*}\bigr)_{i\in\bz}$ does not satisfy the 
condition (E$7^*$). 
In fact, the following fact can be shown. 

\begin{prop}\label{induced filt2}
If $p$ is an odd prime, then for $s=0,1,2,\dots$ and $\varepsilon=0,1$, 
$$\biggl\{x_{21}^{\varepsilon}\prod\limits_{j\ge2}x_{j+1\,j}^{m_j}\bigg|\;
\sum\limits_{j\ge 2}m_jp^{j-2}=s\biggr\}$$
is a basis of $E_{2s+\varepsilon}^{2s(p-1)+\varepsilon}A_{(p)}^*$. 
For $s=0,1,2,\dots$, 
$$\biggl\{\prod\limits_{j\ge1}x_{j+1\,j}^{m_j}\bigg|\;
\sum\limits_{j\ge 2}m_j2^{j-1}=s\biggr\}$$
is a basis of $E_s^sA_{(2)}^*$. 
\end{prop}

\appendix

\section{Appendix}
\setobjecttype{App}\label{sec:sec5}

Here we make an observation on the group scheme represented by the dual 
Steenrod algebra $\ca_{p*}$. 

Let $\wtilde\ca_{p*}$ be the polynomial part of $\ca_{p*}$ 
(hence $\wtilde\ca_{2*}=\ca_{2*}$). 
G Nishida observed that $\wtilde\ca_{p*}$ represents the functor 
$\wtilde\varGamma \co \alg_{\bff_p}^*\to\gr$ defined by 
$$\wtilde\varGamma(R^*)=\{f(X)\in R^*\llbracket
X\rrbracket^{-2}|\,f(X+Y)=f(X)+f(Y),\qua
f(0)=0,\qua f'(0)=1\},$$ 
that is, $\wtilde\varGamma(R^*)$ is the group of strict automorphisms 
of 
the additive formal group law $G_a$ over $R^*$. 
(We regard $R^*\llbracket X\rrbracket$ as a graded ring with $\deg X=-2$.) 

In fact, for a morphism $\varphi \co \wtilde\ca_{p*}\to R^*$ of graded 
rings, put 
$$f_{\varphi}(X)=\sum_{i\ge0}\varphi(\xi_i)X^{p^i} \qua (\xi_0=1).$$
Then, it follows from Milnor \cite[Theorem 3]{Mi} that the correspondence 
$\varphi\mapsto f_{\varphi}(X)$ gives a natural equivalence 
$h_{\wtilde\ca_{p*}}\to\wtilde\varGamma$. 

This fact also has a geometric explanation as follows. 
Let $\alpha \co MU_*\to\bff_p$ be the map that classifies the additive formal 
group law over $\bff_p$. 
Then, the pull-back of the groupoid scheme represented by the Hopf 
algebroid 
$(MU_*,MU_*MU)$ along $h_{\alpha} \co h_{\bff_p}\to h_{MU_*}$ is the 
stabilizer 
group scheme of the additive formal group law and it is represented by 
$\bff_p\otimes_{MU_*}MU_*MU\otimes_{MU_*}\bff_p$ (Yamaguchi \cite{el}). 
Since $\alpha$ factors through the canonical map $MU_*\to BP_*$, 
$\bff_p\otimes_{MU_*}MU_*MU\otimes_{MU_*}\bff_p$ is isomorphic to 
$\bff_p\otimes_{BP_*}BP_*BP\otimes_{BP_*}\bff_p\cong\wtilde\ca_{p*}$. 

We assume that $p$ is an odd prime below. 
Define a functor $\varGamma \co \alg_{\bff_p}^*\to\gr$ as follows. 
For $R^*\in\ob\alg_{\bff_p}^*$, we consider an object 
$R^*[\varepsilon]/(\varepsilon^2)$ ($\deg\,\varepsilon=-1$) of 
$\alg_{\bff_p}^*$. 
Let $\varGamma(R^*)$ be the set of automorphisms $f:G_a\to G_a$ over 
$R^*[\varepsilon]/(\varepsilon^2)$ such that $f'(0)-1\in(\varepsilon)$. 
The group structure of $\varGamma(R^*)$ is given by the composition of 
automorphisms. 
If $\varphi \co R^*\to S^*$ is a homomorphism of graded algebras, 
$\varGamma(\varphi) \co \varGamma(R^*)\to\varGamma(S^*)$ maps 
$f(X)=\sum_{i\ge0}(a_i+b_i\varepsilon)X^{p^i}$ to 
$\sum_{i\ge0}(\varphi(a_i)+\varphi(b_i)\varepsilon)X^{p^i}$. 

\begin{prop}\label{rep}
The affine group scheme $h_{\ca_{p*}}$ represented by $\ca_{p*}$ is 
isomorphic 
to $\varGamma$. 
\end{prop}

\begin{proof}
We define a natural transformation $\theta\co \smash{h_{\ca_{p*}}}\to\varGamma$ as 
follows. 
For $R^*{\in}\ob\smash{\alg_{\bff_p}^*}$ and $\varphi\in h_{\ca_{p*}}(R^*)$, we set 
$\theta_R^*(\varphi)=\sum_{i\ge0}(\varphi(\xi_i)+
\varphi(\tau_i)\varepsilon)X^{p^i}$. 
It follows from Milnor \cite[Theorem 3]{Mi} that $\theta$ is a natural 
transformation. 
We can verify easily that $\theta$ is a natural equivalence. 
\end{proof}

Thus $\varGamma$ is regarded as a closed subscheme of $U_{\infty}$. 

\bibliographystyle{gtart}
\bibliography{link}

\end{document}